\documentclass{elsarticle}
\usepackage{graphicx}
\usepackage{amscd,amsmath,amstext,amsfonts,amsbsy,amssymb,amsthm,eufrak}
\usepackage{hyperref}
\usepackage[toc,page]{appendix}
\usepackage{chngcntr}

\newtheorem{theo}{Theorem}
\newtheorem{lemm}[theo]{Lemma}

\newtheorem{prop}[theo]{Proposition}
\newtheorem{conj}[theo]{Conjecture}
\newdefinition{defi}{Definition}

\newdefinition{rema}{Remark}

\makeatletter
\def\ps@pprintTitle{
 \let\@oddhead\@empty
 \let\@evenhead\@empty
 \def\@oddfoot{}%
 \let\@evenfoot\let\@oddfoot }
\makeatother

\begin{document}

\title{Numerical analysis of the rescaling method for parabolic problems with blow-up in finite time}
\author{V. T. Nguyen}
\ead{vtnguyen@math.univ-paris13.fr}
\address{Universit\'e Paris 13, Sorbonne Paris Cit\'e,\\ LAGA, CNRS (UMR 7539), F-93430, Villetaneuse, France.}

\begin{abstract}
In this work, we study the numerical solution for parabolic equations whose solutions have a common property of  blowing up in finite time and the equations are invariant under the following scaling transformation 
$$u \mapsto u_\lambda(x,t):= \lambda^{\frac{2}{p-1}}u(\lambda x, \lambda^2 t).$$
For that purpose, we apply the rescaling method proposed by Berger and Kohn \cite{BKcpam88} to such problems. The convergence of the method is proved under some regularity assumption. Some numerical experiments are given to derive the blow-up profile verifying henceforth the theoretical results.
\end{abstract}

\maketitle

\section{Introduction}
We study the solution of the following parabolic problem
\begin{equation} \label{equ:semiHeatEqu}
\left\{\begin{array}{lll}
u_t(x,t) &=  u_{xx}(x,t) + g(u, u_x),\quad & \text{in}\quad \Omega \times (0,T),\\
u(x,t)&= 0 \quad & \text{on}\quad \partial \Omega \times [0,T),\\
u(x,0) &= u_0(x), \quad & \text{on}\quad \bar{\Omega}.
\end{array}
\right.
\end{equation}
where $u(t): x \in \Omega \mapsto u(x,t) \in \mathbb{R}$, $p > 1$. The function $g$ is given by
\begin{equation*}\label{equ:funcg}
g(u,u_x) = |u|^{p-1}u + \beta|u_x|^q, \quad \text{with} \quad q = \frac{2p}{p+1},
\end{equation*}
for some $\beta \in \mathbb{R}$. This equation can be viewed as a population dynamic model (see \cite{SOUmmas96} for an example).\\
We also consider the complex Ginzburg-Landau equation, 
\begin{equation} \label{equ:GL}
\left\{\begin{array}{lll}
u_t(x,t) &= (1 + \imath \gamma)u_{xx} + (1 + \imath \delta) |u|^{p-1}u,\quad & \text{in}\quad \Omega \times (0,T),\\
u(x,t)&= 0 \quad & \text{on}\quad \partial \Omega \times [0,T),\\
u(x,0) &= u_0(x), \quad & \text{on}\quad \bar{\Omega}.
\end{array}
\right.
\end{equation}
where $u(t): x \in \Omega \to u(x,t) \in \mathbb{C}$, $p> 1$ and the constants $\gamma, \delta $ are real. This equation appears in various physical situations.  An example is the theory of phase transitions and superconductivity. We refer to Popp et al. \cite{POPphd98} and the references therein for the physical background.\\

In both problems, $\Omega$ is a bounded interval and $u_0: \bar{\Omega} \to \mathbb{R}$ is a given initial value that belongs to $H$ where  $H \equiv W^{1,\infty}(\Omega)$ for equation \eqref{equ:semiHeatEqu} and $H \equiv L^{\infty}(\Omega)$ for equation \eqref{equ:GL}. In particular, we consider $\Omega = (-1, 1)$ and $u_0$ is positive, nontrivial, smooth and verifies $u_0(-1) = u_0(1) = 0$; in addition, $u_0$ is symmetric and nondecreasing on the interval $(-1,0)$. Thanks to a fixed-point argument, the Cauchy problem for equation \eqref{equ:semiHeatEqu} can be solved in $W^{1,\infty}(\Omega)$, locally in time. For equation \eqref{equ:GL}, we solve it in $L^{\infty}(\Omega)$. Then, it is easy to see that the maximal solution is either global in time, or exists only for $t \in [0,T)$ for some $T > 0$. In that case, the solution 
blows up in finite time $T$, namely,
$$\lim_{t \to T} \|u(t)\|_H = + \infty,$$
and $T$ is called the blow-up time of $u(t)$.

When $\beta = 0$, the theoretical part for equation \eqref{equ:semiHeatEqu} is largely well-understood. The literature on the subject is huge, so we refer the reader to the book by Souplet and Quittner \cite{QSbook07}. When $\beta \ne 0$ and $q > 0$, less is known about blow-up for equation \eqref{equ:semiHeatEqu}. As a matter of fact, we loose the gradient structure, and energy methods break down. We keep however a maximum principal. We have several contributions on the subject by \cite{CWjma89}, \cite{SOUmmas96}, \cite{STWiumj96} and \cite{EZsema11}. Note that our choice $q = \frac{2p}{p+1}$ is critical in the sense that it is the only choice that makes equation \eqref{equ:semiHeatEqu} invariant under the dilation given in \eqref{equ:ScaPro} below. As for equation \eqref{equ:GL}, when $\gamma \ne \delta$, we have no gradient structure nor maximum principle. Therefore, classical methods cannot be applied here. Up to our knowledge, there are not many papers on this subject, apart from the paper of Popp et al. \cite{POPphd98} and the paper by Masmoudi and Zaag \cite{MZjfa08} who construct a stable blow-up solution. There is also a paper by Cazenave, Dickstein and Weissler \cite{CDWjma13} when $\gamma = \delta$ (note that in this case, there is a Lyapunov functional).
\bigskip

In comparison with the theoretical aspects, the numerical analysis of blow-up has received little attention, particularly on the numerical blow-up profile. For other numerical aspects related to sufficient blow-up conditions, the blow-up rate, the blow-up time and the blow-up set, there are several studies for \eqref{equ:semiHeatEqu} in the case $\beta = 0$. The first work on this problem was done in \cite{NAKamo76, NUna77} by using the finite difference and finite element method on a uniform spatial mesh. For sufficient blow-up conditions, the solution of semi or full-discretized equation blowing up in finite time was established in \cite{ALManm96, ALManm01}, \cite{CHEjut86, CHEhmj92}, \cite{DERdcds98}, \cite{NAKamo76} and \cite{NBams11}. For the numerical blow-up rate, there is a series of studies by \cite{FGRmmmas02, AFGRdcds02, FGRnmpde04, GROcomp06}, \cite{HKjcam06} and  \cite{NBams11}. Those papers gave the relation between the discretized problem and the continuous ones. For the numerical convergence of the blow-up time, it was investigated in \cite{ALManm96, ALManm98, USHprims00, ADRcomp02}, \cite{HKjcam06} and  \cite{NBams11}. On numerical blow-up sets, we would like to mention the works in \cite{FGRpams02, GRjcam01, FGRnmpde04} and \cite{ABNjam08}. Up to our knowledge, there are not many papers on the numerical blow-up profile, apart from the paper of Berger and Kohn \cite{BKcpam88} who already obtained  very good numerical results on this subject. There is also the work of Baruch et al. \cite{BFGpd10} studying standing-ring solutions.

For this reason, we will rely on the rescaling method suggested in \cite{BKcpam88} to obtain a numerical solution for the equations mentioned above. This algorithm fundamentally relies on the scale invariance of equation  \eqref{equ:semiHeatEqu} and \eqref{equ:GL}: if $u$ a solution of \eqref{equ:semiHeatEqu} (or \eqref{equ:GL}), then for all $\lambda > 0$, the function $u_\lambda$ given by
\begin{equation}\label{equ:ScaPro}
u_\lambda(\xi, \tau) = \lambda^{\frac{2}{p-1}}u(\lambda \xi, \lambda^2 \tau),
\end{equation}
is also a solution of \eqref{equ:semiHeatEqu} (or \eqref{equ:GL}). This property allows to make a zoom of the solution when it is close to the singularity, still keeping the same equation. Our aim is to give a numerical confirmation for the theoretical profile of the semilinear heat equation  \eqref{equ:semiHeatEqu} in the case $\beta = 0$ (already done in \cite{BKcpam88}) and especially the complex Ginzburg-Landau equation \eqref{equ:GL} which has never been done earlier numerically, and is quite challenging. In the case $\beta \ne 0$ in equation \eqref{equ:semiHeatEqu}, we give a numerical answer to the question of the blow-up profile, where no theoretical is available. This way, our numerical result gives use to new conjecture.\\

The paper is organized as follows: In section \ref{sec:2}, we give some theoretical framework on the study. Section \ref{sec:3} presents the approximation scheme and the rescaling algorithm. The convergence of the numerical solution for problem \eqref{equ:semiHeatEqu} is proved in section \ref{sec:4}. In the last section, we give some numerical experiments to confirm the theoretical results.\\

\noindent \textbf{Acknowledgement:} The author is grateful to L. El Alaoui and H. Zaag for helpful suggestions and remarks during the preparation of this paper.

\section{The theoretical framework} \label{sec:2}

\textbf{Equation \eqref{equ:semiHeatEqu} in case $\beta = 0$}: The existence of blow-up solution for equation \eqref{equ:semiHeatEqu} has been proved by several authors (\cite{Friams65, FUJsut66, LEVarma73, BALjmo77}). We have lots of results concerning the behavior of the solution $u$ of \eqref{equ:semiHeatEqu} at blow-up time, near blow-up points (\cite{GKcpam85, GKiumj87, GKcpam89}, \cite{FKcpam92, FLaihn93},  \cite{HVaihn93, HVasnsp92, HVasnsp92, VELcpde92, VELtams93} and  \cite{MZdm97, MZgfa98}). This study has been done through the introduction for each $a \in \Omega$ ($a$ may be a blow-up point of $u$ or not) the following \emph{similarity variables:} 
\begin{equation}  \label{equ:simiWU}
w_{a,T}(y, s) = (T - t)^{\frac{1}{p-1}}u(x,t), \quad
y = \frac{x - a}{\sqrt{T - t}}, \quad s = -\log{(T -t)},
\end{equation}
and $w_{a,T} = w$ solves a new parabolic equation in $(y,s)$: for all $s\ge -\log{T}$ and $y \in D_{a,s}$, $D_{a,s} = \{y \in \mathbb{R}| a + ye^{-s/2} \in \Omega\}$,
\begin{equation}\label{equ:equinw}
\partial_s w = \Delta w  - \frac{1}{2}y \cdot \nabla w - \frac{w}{p-1} + |w|^{p-1}w.
\end{equation}
Studying solutions of \eqref{equ:semiHeatEqu} near blow-up is therefore equivalent to analyzing large-time asymptotics of solutions of \eqref{equ:equinw}. Each result for $u$ has an equivalent formulation in terms of $w$.\\

\noindent One of the main results which is established in \cite{GKiumj87, GKcpam89} is that $a$ is a blow-up point if and only if 
\begin{equation*}
\lim_{t \to T} (T - t)^{\frac{1}{p-1}}u(a + y \sqrt{T-t}, t) = \pm\kappa,
\end{equation*}
uniformly in $|y| \le C$, where $\kappa = (p-1)^{-\frac{1}{p-1}}$.\\

\noindent In \cite{GPans85, GPans86}, the authors used a formal argument adapted  from \cite{HSBjfm72}  to derive  the ansatz 
\begin{equation}\label{equ:ans}
u(x,t) \sim (T - t)^{-\frac{1}{p-1}} \left(p-1 + \frac{(1-p)^2}{4p}\frac{(x-a)^2}{(T -t)|\log{(T - t)}|}\right)^{-\frac{1}{p-1}}.
\end{equation}
This ansatz has been proved in \cite{VELcpde92, BKnon94, MZgfa98} for some examples of initial data. More precisely, $w$ has a limiting profile in the variable $z = \frac{y}{\sqrt{s}}$ (see \cite{MZdm97, MZgfa98}, \cite{VELcpde92,HVaihn93}), in the sense that 
\begin{equation}\label{equ:preproSLHn}
\sup_{|y| \leq K\sqrt{s}} \left| w(y,s) - f \left(\frac{y}{\sqrt{s}}\right)\right| \to 0 \quad \text{as } s \to +\infty,
\end{equation}
for any $K > 0$, where
\begin{equation}\label{equ:defFshe}
f(z) = \left(p-1 + \frac{(p-1)^2}{4p}|z|^2 \right)^{-\frac{1}{p-1}}.
\end{equation}
The profile \eqref{equ:defFshe} is stable under perturbations of initial data, other profiles are possible but they are suspected to be unstable (see \cite{MZdm97, FZnon00, FMZma00}). Note that Herrero and Vel\'azquez proved the genericity of the behavior \eqref{equ:ans} in  \cite{HVasnsp92} and \cite{HVasps92} in one space dimension. \\

\noindent \textbf{Equation \eqref{equ:semiHeatEqu} in case $\beta \ne 0$}: When $\beta \in (-2,0)$, in \cite{STWiumj96} (see also \cite{SOUmmas96, CWjma89}), the authors proved the existence of a non-trivial backward self-similar solution which blows up in finite time, only at one point and described the asymptotic behavior of its radially symmetric profile. More precisely, they showed the existence of a solution of \eqref{equ:semiHeatEqu} of the form 
\begin{equation}\label{equ:sesi}
u(x,t) = (T - t)^{-\frac{1}{p-1}}v\left(\frac{x}{\sqrt{T-t}}\right),
\end{equation}
where $v$ satisfies for all $\xi \in \mathbb{R}$,
\begin{equation*}
\Delta v(\xi) + \beta |\nabla v(\xi)|^q - \left[\frac{\xi}{2} \cdot \nabla v(\xi) + \frac{1}{p+1}v(\xi) \right] + |v(\xi)|^{p-1}v(\xi) = 0.
\end{equation*}
Note that this type of behavior does not hold when $\beta = 0$. Indeed, from Giga and Kohn \cite{GKcpam85}, we know that the only solutions of the form \eqref{equ:sesi} are $0$ and $\pm \kappa (T - t)^{-\frac{1}{p-1}}$ with $\kappa = (p-1)^{-\frac{1}{p-1}}$. \\

We wonder however whether equation \eqref{equ:semiHeatEqu} has solutions which behave like the solution of the case $\beta = 0$, namely such that, for all $K >0$,
\begin{equation}\label{equ:preproGran}
\sup_{|z|<K}|(T-t)^{1/(p-1)}u(x,t) - \bar{f_\beta}(z)| \to 0, \quad \text{as}\quad  t \to T,
\end{equation}
where $z= \frac{x}{\sqrt{(T-t)|\log(T-t)|}}$,
\begin{equation}\label{equ:defFgra}
\bar{f_\beta}(z) = \left(p-1+b(\beta)|z|^2\right)^{-\frac{1}{p-1}}, \quad \text{with}\quad b(0) = \frac{(p-1)^2}{4p},
\end{equation}
or in \emph{similarity variables} defined in \eqref{equ:simiWU},
\begin{equation}\label{equ:preproFgra}
\sup_{|z|< K}|w(y,s) - \bar{f_\beta}(z)| \to 0, \quad \text{as} \quad s \to \infty.
\end{equation}
Up to our knowledge, there is no theoretical answer to this equation. We answer it positively through a numerical method in this paper (see Section \ref{sec:5.2} below).\\

\noindent\textbf{The complex Ginzburg-Landau equation}: In \cite{MZjfa08}, Masmoudi and Zaag constructed the first solution to equation \eqref{equ:GL} which blows up in finite time $T$ only at one blow-up point and gave a sharp description of its blow-up profile. Furthermore, they showed the stability of that solution with respect to pertubation in initial data. Their result extends the previous result of Zaag \cite{ZAAihn98} done for $\gamma =0$. More precisely, they used the following self-similar transformation of equation \eqref{equ:GL}: 
\begin{equation}\label{equ:tranGL}
w_{a,T}(y, s) = (T - t)^{\frac{1 + \imath \delta}{p-1}}u(x,t), \quad
y = \frac{x - a}{\sqrt{T - t}}, \quad s = -\log{(T -t)},
\end{equation}
and then $w(y,s)$ satisfies the following equation:
\begin{equation*}
\partial_s w = (1 + \imath \gamma)\Delta w  - \frac{1}{2}y \cdot \nabla w - \frac{1 + \imath \delta}{p-1}w + (1 + \imath \delta)|w|^{p-1}w.
\end{equation*}
Their main result is the following: for any $(\delta, \gamma) \in \mathbb{R}^2$ such that $p - \delta^2 - \gamma\delta(p+1) > 0$, equation \eqref{equ:GL} has a solution $u(x,t)$ blowing up in finite time $T$ only at a point $a \in \mathbb{R}$. Moreover, if $w = w_{a,T}$ defined in \eqref{equ:tranGL}, then
\begin{equation}\label{equ:profileGL}
\left\| |s|^{-\imath \mu} w(y,s) - \tilde{f}_{\delta, \gamma}(z) \right\|_{L^\infty} \leq \frac{C}{1 + \sqrt{|s|}}, \quad z = \frac{y}{\sqrt{s}},
\end{equation}
where $\mu =  - \frac{2\gamma b(\delta, \gamma)}{(p-1)^2} (1 + \delta^2), \; b(\delta, \gamma) = \frac{(p-1)^2}{4(p - \delta^2 - \gamma\delta(p+1))}$ and
\begin{equation}\label{equ:funcfpGL}
\tilde{f}_{\delta, \gamma}(z) = \left(p - 1 + b(\delta, \gamma)|z|^2 \right)^{-\frac{1 + \imath \delta}{p-1}}.
\end{equation}
\begin{rema}
We remark that equation \eqref{equ:GL} is rotation invariant. Therefore, $e^{\imath \theta} \tilde{f}_{\delta, \gamma}$ is also an asymptotic profile of the solution of \eqref{equ:GL} with $\theta \in \mathbb{R}$. 
\end{rema}
\begin{rema}
In our paper, we give the first numerical computation of this result. Note that the stability result of \cite{MZjfa08} concerning that solution makes it visible in numerical simulations.
\end{rema}
\section{The numerical method} \label{sec:3}
\noindent In this section, we recall the rescaling algorithm introduced
in \cite{BKcpam88}. 
\subsection{The numerical scheme} 
We first give an Euler approximation of \eqref{equ:semiHeatEqu} and \eqref{equ:GL}. Let $I$ be a positive integer and let us discretize the domain $\Omega = (-1,1)$ by the grid $x_i = ih$ where $-I \leq i \leq I$ and $h = \frac{1}{I}$. Let $\tau > 0$ be a time step and $n \geq 0$ be a positive integer. Then, we set $t_n = n \tau$. In what follows, the lowercase letter denotes the exact values, whereas the capital letter denotes its approximation, for example, we write $u_{i,n} \equiv u(x_i, t_n)$ and $U_{i,n}$ the approximation of $u(x_i, t_n)$. In the following, the notation $\mathbf{U}_n$ stands for $(U_{-I}, \dots, U_0, \dots, U_I)^T$. In addition, we denote 
$$\delta_t U_{i,n} = \frac{U_{i,n+1} - U_{i,n}}{\tau},$$
\begin{equation}
\delta_x U_{i,n} = \frac{U_{i+1,n} - U_{i-1,n}}{2h},\label{equ:cendiff}
\end{equation}
$$\delta_x^2 U_{i,n} = \frac{U_{i-1,n} - 2U_{i,n} + U_{i+1,n}}{h^2}.$$

\noindent\textbf{Discretization of the semilinear heat equation:}\\
The Euler discretization of \eqref{equ:semiHeatEqu} is defined as follows: for $n \geq 0$ and $-I +1 \leq i \leq I-1$,
\begin{equation}\label{equ:sys}
\left\{ \begin{array}{rcll}
\delta_t U_{i,n} &=& \delta^2_x U_{i,n} + \left| U_{i,n}^{p-1}\right|U_{i,n}  + \beta |\delta_x U_{i,n}|^\frac{2p}{p+1},\\
U_{-I, n} &=& U_{I,n} = 0, 
\end{array}
\right.
\end{equation}
with $ U_{i,0} = \phi_i$ where $\phi_i = u_0(x_i)$. Note that $U_{i,n}$ is defined for all $n \geq 0$ and $-I \leq i \leq I$.\\

\noindent\textbf{Discretization of the Ginzburg-Landau equation:}\\
Let us write the solution  of \eqref{equ:GL} as $u = v
+ \imath w$ and $|u| = \sqrt{v^2 + w^2}$. Then \eqref{equ:GL} can be rewritten as follows:  
\begin{equation}\label{equ:sysGL1}
\left\{
\begin{array}{lll}
v_t &=& v_{xx} - \gamma w_{xx}  + \left(v^2 + w^2\right)^\frac{p-1}{2} (v - \delta w)\\
w_t &=& \gamma v_{xx} + w_{xx} + \left(v^2 + w^2\right)^\frac{p-1}{2}  (\delta v + w)\\
v(x,0) &=& \Re \big(u_0(x)\big), \quad w(x,0) = \Im\big(u_0(x)\big).
\end{array}
\right.
\end{equation}
Denote by $V_{i,n}$ and $W_{i,n}$ approximations of
$v(x_i,t_n)$ and $w(x_i,t_n)$ respectively. On setting $\mathbf{V}_n = (V_{-I,n}, \dots, V_{I,n})^T$, $\mathbf{W}_n = (W_{-I,n}, \dots, W_{I,n})^T$, the Euler scheme approximating the
solution of \eqref{equ:sysGL1} is given below: for $n \geq 0$ and $-I +1 \leq i \leq I-1$,
\begin{equation}\label{equ:sysGL}
\left\{ \begin{array}{rcl}
\delta_t V_{i,n} &=& \delta^2_x V_{i,n} - \gamma \delta^2_x W_{i,n} + R_{i,n}(V_{i,n} - \delta W_{i,n}),\\
\delta_t W_{i,n} &=& \gamma\delta^2_x V_{i,n} + \delta^2_x W_{i,n} + R_{i,n}(\delta V_{i,n} + W_{i,n}),\\
V_{-I, n} &=& V_{I,n} = W_{-I, n} = W_{I,n} = 0,\quad 
\end{array}
\right.	
\end{equation}
with $V_{i,0} = \Re\big(\phi_i\big), W_{i,0} = \Im\big(\phi_i\big)$ where $\phi_i = u_0(x_i)$ and $R_{i,n} = (V_{i,n}^2 + W_{i,n}^2)^{\frac{p-1}{2}}$.
\begin{rema} By Taylor expansion, one can show that the central difference approximation given in \eqref{equ:cendiff} is second-order accurate. Therefore, both difference schemes \eqref{equ:sys} and \eqref{equ:sysGL} are first-order accurate in time and second-order in space. 
\end{rema}

\noindent In what follows, let $\mathbf{a} = (a_{-I}, \dots, a_0, \dots, a_I)^T$, denote $\|\mathbf{a}\|_\infty = \max\limits_{i} |a_i|$. We say that $\mathbf{a}$ is positive if each component of $\mathbf{a}$ is positive and write $\mathbf{a} > 0$. Similar notations $\geq, \leq , <$ can be defined.

\subsection{The rescaling method} \label{sec:resMe}
For the sake of clarity, we present the rescaling method in \cite{BKcpam88}, only for the approximation of the semilinear heat equation \eqref{equ:sys}. Straightforward adaptations allos to derive it for the Ginzburg-Landau equation approximated in \eqref{equ:sysGL}. \\
We first introduce some notations: 
\begin{enumerate}[$\diamond$]
\item $\lambda < 1$ is a scaling factor such that $\lambda^{-1}$ is a small positive integer.
\item $M$ is a maximum amplitude before rescaling.
\item $\alpha$ is a parameter controlling the width of the interval to be rescaled.
\item $u^{(k)}(\xi_k, \eta_k)$ is the $k$-th \emph{rescaled solution} defined in space-time variables $(\xi_k, \eta_k)$. If $k = 0$,  $u^{(0)} (\xi_0, \eta_0) \equiv u(x,t), \, (\xi_0, \eta_0) \equiv (x,t)$.
\item $h_k, \tau_k$ denote the space and time step used to approximate $u^{(k)}$.
\item $U^{(k)}_{i,n}$ is an approximation value of $u^{(k)}(\xi_{k,i}, \eta_{k,n})$ where $\xi_{k,i} = i h_k$ 
and $\eta_{k,n} = n \tau_k$.
\end{enumerate}
Let $\{(x_i, t_n , F_{i,n})|  -I \leq i \leq I, 0 \leq n \leq N\}$ be a set of data points, we associate the function $F_{h,\tau}$ which is a piecewise linear approximation in both space and time such that $F_{h,\tau}(x_i,t_n) = F_{i,n}$ and for all $(x,t) \in (x_i, x_{i+1})\times (t_n, t_{n+1})$, 
\begin{align}
F_{h,\tau}(x,t) &= \frac{1}{h\tau}\left[F_{i,n}(x_{i+1} - x)(t_{n+1} - t) + F_{i+1,n}(x - x_i)(t_{n+1} - t) \right] \nonumber \\
& + \frac{1}{h\tau}\left[F_{i,n+1}(x_{i+1} - x)(t - t_n) + F_{i+1,n+1} (x - x_i)(t - t_n)\right].\label{equ:intST}
\end{align}
At some points, we may use the notation $F_{h,n}(x) \equiv F_{h,\tau}(x,t_n)$ for a given $t_n$ and $F_{i,\tau}(t) \equiv F_{h,\tau}(x_i, t)$ for a given $x_i$.\\

\noindent We now recall the rescaling method introduced in \cite{BKcpam88}. \\
The solution of \eqref{equ:sys} is integrated until getting the first time step $\mathbf{n}_0$ such that $\|\mathbf{U}_{\mathbf{n}_0}\|_\infty \geq M$.
Then we find out a value $\tau^*_0$ satisfying 
\begin{equation*}
(\mathbf{n}_0 - 1)\tau \leq \tau_0^* \leq \mathbf{n}_0 \tau \quad \text{and} \quad \|\mathbf{U}_{h,\tau}\left(\cdot,\tau_0^*\right)\|_\infty = M,
\end{equation*}
and two grid points $x_{i_0^-}, x_{i_0^+}$, with $i_0^+, i_0^- \in \{-I, \dots, 0, \dots, I\}$, such that
\begin{equation*}
\left\{ \begin{array}{rcl}
\mathbf{U}_{h,\tau}(x_{i_0^- -1},\tau^*_0) &< \alpha M &\leq \mathbf{U}_{h,\tau}(x_{i_0^-},\tau^*_0),\\
\mathbf{U}_{h,\tau}(x_{i_0^+ +1},\tau^*_0) &< \alpha M &\leq \mathbf{U}_{h,\tau}(x_{i_0^+},\tau^*_0).
\end{array}
\right.
\end{equation*}

\noindent On the interval $(x_{i_0^-}, x_{i_0^+})$  and for $t \geq \tau_0^*$, we refine the mesh by a factor $\lambda$ in space and $\lambda^2$ in time. More precisely, we introduce
\begin{equation*}
u^{(1)}(\xi_1, \eta_1) = \lambda^{\frac{2}{p-1}}u(\lambda \xi_1, \tau_0^* + \lambda^2 \eta_1),
\end{equation*} 
which is also a solution of equation \eqref{equ:semiHeatEqu}, thanks to the scale invariance property stated after \eqref{equ:ScaPro}. From a numerical point of view, it is important to use for $u^{(1)}$ the same discretization  as for $u$. Let $h_1$ be the space discretization step and $\tau_1$ be the time discretization step, then we need to set $h_1 = h$ and $\tau_1 = \tau$ to use the same scheme \eqref{equ:sys} for approximating 
$u^{(1)}$. In other words, the approximation of $u$ on the
interval $(x_{i_0^-}, x_{i_0^+})$ with the steps $\lambda h,
\lambda^2 \tau$ is equivalent to the approximation of
$u^{(1)}$ on the interval $\lambda^{-1}\left(x_{i_0^-},
x_{i_0^+}\right)$ by using $h$ and $\tau$ as
discretization parameters. \\

\noindent Let $I_1 = \lambda^{-1}i_0^+$ and $\mathbf{U}^{(1)}_n =
\left(U^{(1)}_{-I_1,n}, \dots, U^{(1)}_{0,n}, \dots,
U^{(1)}_{I_1,n}\right)^T$ be an approximation of $u^{(1)}$ at time $\eta_{1,n}$. Then, $\mathbf{U}^{(1)}_{n+1}$ solves the following equations: for all $n \geq 0$, $i$ between $-I_1 + 1$ and $I_1 - 1$,
\begin{equation*}
\left\{ \begin{array}{lcl}
\delta_t U^{(1)}_{i,n} &=& \delta^2_x U^{(1)}_{i,n} + \left|U^{(1)}_{i,n}\right|^{p-1}U^{(1)}_{i,n} + \beta \left|\delta_x U^{(1)}_{i,n}\right|^\frac{2p}{p+1},\\
& \\
U^{(1)}_{I_1, n} &=& U^{(1)}_{-I_1, n} = \psi_n^{(1)}, \quad U^{(1)}_{i,0} = \phi_i^{(1)},
\end{array}
\right.
\end{equation*}
where 
\begin{align}
\psi_n^{(1)} & = \lambda^{\frac{2}{p-1}}\mathbf{U}_{h,\tau}(x_{i_0^+}, \tau_0^* + \lambda^2 n \tau), \quad n \geq 0, \label{equ:psi1}\\
\phi_i^{(1)} & = \lambda^{\frac{2}{p-1}}\mathbf{U}_{h,\tau}(\lambda \xi_{1,i}, \tau_0^*), \quad -I_1 \leq i \leq I_1. \label{equ:phi1}
\end{align}
We stop the computation of $\mathbf{U}^{(1)}$ at the first
time level $\eta_{1,\mathbf{n}_1}$ ($\mathbf{n}_1 \geq 1$) such that $\left\|\mathbf{U}^{(1)}_{\mathbf{n}_1} \right\|_\infty \geq M$. After that, we determine $\tau_1^*$
and two grid points $\xi_{1,i_1^-}, \xi_{1,i_1^+}$ where
$i_1^-, i_1^+ \in \{-I_1, \dots, 0, \dots, I_1\}$ by  
\begin{equation*}
\left\{\begin{array}{l}
(\mathbf{n}_1 - 1)\tau_1  \leq \tau_1^* \leq \mathbf{n}_1 \tau_1\\
\left\|\mathbf{U}_{h,\tau}^{(1)}\left(\cdot,\tau_1^*\right)\right\|_\infty = M,
 \end{array} \right. 
 \;\text{and} \;
 \left\{ \begin{array}{rcl}
\mathbf{U}_{h,\tau}^{(1)}(\xi_{1,i_1^- -1},\tau^*_1)&<\alpha M&\leq\mathbf{U}_{h,\tau}^{(1)}(\xi_{1,i_1^-},\tau^*_1),\\
\mathbf{U}_{h,\tau}^{(1)}(\xi_{1,i_1^+ +1},\tau^*_1)&< \alpha M&\leq\mathbf{U}_{h,\tau}^{(1)}(\xi_{1,i_1^+},\tau^*_1).
\end{array}
\right.
\end{equation*}
We remark that the computation of $\mathbf{U}^{(1)}$
requires an initial and a boundary conditions. The initial
data conditions are already obtained by \eqref{equ:phi1}. It remains to
focus on the boundary condition \eqref{equ:psi1}. Both
$\mathbf{U}$ and $\mathbf{U}^{(1)}$ are stepped forward
independently, each on its own grid. A single time step of
$\mathbf{U}$ corresponds to $\lambda^{-2}$ time steps of
$\mathbf{U}^{(1)}$. Therefore, the linear interpolation in
time of $\mathbf{U}$ is used to find the boundary values of
$\mathbf{U}^{(1)}$. After stepping forward
$\mathbf{U}^{(1)}$ $\lambda^{-2}$ times, the values of
$\mathbf{U}$ at grid points on the interval $(x_{i^-_0},
x_{i^+_0})$ are modified to better with the fine grid
solution $\mathbf{U}^{(1)}$. On the interval where
$\mathbf{U}^{(1)} > \alpha M$, the entire procedure is
repeated, yielding $\mathbf{U}^{(2)}$, and so forth.\\ 

The $(k+1)$-st rescaled solution $u^{(k+1)}$ is introduced when $\eta_k$ reaches a value $\tau_k^*$ satisfying
\begin{equation}\label{equ:inittauk}
(\mathbf{n}_k - 1)\tau_k \leq \tau_k^* \leq \mathbf{n}_k \tau_k, \quad \mathbf{n}_k > 0 \quad \text{and} \quad  \left\|\mathbf{U}^{(k)}\left(\cdot,\tau_k^*\right) \right\|_\infty  = M.
\end{equation}
The interval $(\xi_{k,i_k^-}, \xi_{k,i_k^+})$ to be rescaled satisfies
\begin{equation*}
\left\{ \begin{array}{rcl}
\mathbf{U}^{(k)}_{h, \tau} (\xi_{k,i_k^- -1},\tau^*_k) &< \alpha M &\leq \mathbf{U}^{(k)}_{h, \tau} (\xi_{k,i_k^-},\tau^*_k),\\
\mathbf{U}^{(k)}_{h, \tau} (\xi_{k,i_k^+ +1},\tau^*_k) &< \alpha M &\leq \mathbf{U}^{(k)}_{h, \tau} (\xi_{k,i_k^+},\tau^*_k).
\end{array}
\right.
\end{equation*}
The solution $u^{(k+1)}$ is related to $u^{(k)}$ by 
\begin{equation}\label{equ:u_kp1}
u^{(k+1)}(\xi_{k+1}, \eta_{k+1}) = \lambda^{\frac{2}{p-1}}u^{(k)}(\lambda \xi_{k+1}, \tau_k^* + \lambda^2 \eta_{k+1}).
\end{equation}
Let $I_{k+1} = \lambda^{-1}i_{k}^+$ and
$$\mathbf{U}^{(k+1)}_n = \left( U^{(k+1)}_{-I_{k+1},n}, \dots, U^{(k+1)}_{0,n},\dots, U^{(k+1)}_{I_{k+1},n}\right)^T$$
be an approximation of $u^{(k+1)}$ at time $\eta_{k+1,n}$. Then $\mathbf{U}^{(k+1)}_{n+1}$ is a solution of the following equations: for all $n \geq 0$, $i$ between $-I_{k+1} +1$ and $I_{k+1} - 1$,
\begin{equation}\label{equ:sysk}
\left\{ \begin{array}{rcl}
\delta_t U^{(k+1)}_{i,n} &=& \delta^2_x U^{(k+1)}_{i,n} +  \left|U^{(k+1)}_{i,n}\right|^{p-1}U^{(k+1)}_{i,n} + \beta \left|\delta_x U^{(k+1)}_{i,n}\right|^\frac{2p}{p+1},\\
&\\
U^{(k+1)}_{I_k, n} &=& U^{(k+1)}_{-I_k, n}= \psi_n^{(k+1)}, \quad U^{k+1}_{i,0} = \phi_i^{(k+1)},
\end{array}
\right.
\end{equation}
where 
\begin{align}
\psi_n^{(k+1)} & = \lambda^{\frac{2}{p-1}}\mathbf{U}^{(k)}_{h,\tau}(\xi_{k,i_k^+}, \tau_k^* + \lambda^2 n \tau), \quad n \geq 0,\label{equ:psik1}\\
\phi_i^{(k+1)} &= \lambda^{\frac{2}{p-1}}\mathbf{U}^{(k)}_{h,\tau}(\lambda\xi_{k+1,i}, \tau_k^*), \quad -I_{k+1} \leq i \leq I_{k+1}. \label{equ:phik1}
\end{align}

We step forward $\mathbf{U}^{(k+1)}$ on  the interval $\lambda^{-1}(\xi_{k,i_k^-},\xi_{k,i_k^+})$ with the space step $h_{k+1}$ and time step $\tau_{k+1}$. Here, we set $h_{k+1} = h_{k} = \dots = h$ and $\tau_{k+1} = \tau_{k} = \dots = \tau$ to use the same scheme as for $\mathbf{U}^{(k)}, \mathbf{U}^{(k-1)}, \dots ,\mathbf{U}$. The initial data of \eqref{equ:sysk} is given in \eqref{equ:phik1}. For the boundary data of \eqref{equ:sysk}, it is obtained by using the linear interpolation in time of $\mathbf{U}^{(k)}$ given in \eqref{equ:psik1}. Hence, we step forward independently the previous solutions $\mathbf{U}^{(k)}, \mathbf{U}^{(k-1)}, \dots$ each one on its own grid. Previously, $\mathbf{U}^{(k)}$ is stepped forward once every $\lambda^{-2}$ time steps of $\mathbf{U}^{(k+1)}$, $\mathbf{U}^{(k-1)}$ once every $\lambda^{-4}$ time steps of $\mathbf{U}^{(k+1)}, \dots$. After $\lambda^{-2}$ time steps of $\mathbf{U}^{(k+1)}$, the values of $\mathbf{U}^{(k)}$ on the interval which has been refined need to be updated to fit with the calculation of $\mathbf{U}^{(k+1)}$; this is performed on $\mathbf{U}^{(k-1)}$ after $\lambda^{-4}$ time steps of $\mathbf{U}^{(k+1)}$ and so forth. We stop the evolution of $\mathbf{U}^{(k+1)}$ when its amplitude reaches the given threshold $M$ and another rescaling can be performed.\\

To make it clearer, we describe the rescaling method by
the following \emph{algorithm}. Assume that we perform up
to the $K$-th \emph{rescaled solution}.
\begin{itemize}
\item[0.] Set up parameters: $M, \lambda, \alpha, h, \tau, I$.
\item[1.] Initial phase: 
\begin{itemize}
\item Forward $\mathbf{U}$ until $\max_{i}U_i \geq M$.
\item Get the values of $\tau_0^*$ and $x_{i_0^-}, x_{i_0+}$.
\end{itemize}
\item[2.] Iterative phase: set $k = 1$, while $k \leq K$ then
\begin{itemize}
\item[(a)] Define a grid for $\mathbf{U}^{(k)}$ on the interval $\lambda^{-1}\left(\xi_{k-1,i_{k-1}^-}, \xi_{k-1,i_{k-1}^+}\right)$.
\item[(b)] Compute the initial data for $\mathbf{U}^{(k)}$ from $\mathbf{U}^{(k-1)}_{h,\tau}(\cdot, \tau_{k-1}^*)$.
\item[(c)] For $i = 0$ to $k-1$: forward $\mathbf{U}^{(i)}$ one step.
\item[(d)] Set $n = 1$.
\item[(e)] While $\max_{i} U_i^{(k)} < M$ then
\begin{itemize}
\item[+] Forward $\mathbf{U}^{(k)}$ one step.
\item[+] Compute the boundary values of $\mathbf{U}^{(k)}$ from $\mathbf{U}^{(k-1)}_{h,\tau}(\xi_{k-1,i_{k-1}^+}, \tau_{k-1}^* + \lambda n \tau_k)$.
\item[+] For $j = 0$ to $k-1$: if $\mod \left(n , \lambda^{-2(j+1)}\right) = 0 $ then
\begin{itemize}
\item[-] Update $\mathbf{U}^{(k - j - 1)}$ on the interval to be rescaled.
\item[-] Forward $\mathbf{U}^{(k - j - 1)}$ one step.
\end{itemize}
\item[+] Set $n = n+1$.
\end{itemize}
\item[(f)] Get the values of $\tau_{k}^*$ and $\xi_{k,i_k^-}, \xi_{k, i_k^+}$.
\item[(g)] For $i =1$ to $k$: update $\mathbf{U}^{k-i}$.
\item[(i)] Set $k = k +1$, $\mathbf{n}_k = n$ and go to step $(a)$.
\end{itemize}
\end{itemize}

\begin{rema} \label{rema:56}
The value of $M$ should be chosen such that the maximum of the initial data of all rescaled solutions are equal. This means that for all $k \geq 0$,
$$
\lambda^{\frac{2}{p-1}}\|u^{(k)}(\tau_k^*)\|_\infty = \|u_0\|_\infty.
$$
Using the fact that $\|u^{(k)}(\tau_k^*)\|_\infty = M$, it yields that $M =\|u_0\|_\infty \lambda^{-\frac{2}{p-1}}$.
\end{rema}
To end this section, we want to give a definition of the numerical solution $\mathbf{U}_{h, \tau}(x,t)$ of the rescaling method. Let $\sigma>0$ small enough, $h >0$ and $\tau >0$ be the space and time step, then, for each $(x,t)\in[-1,1] \times [0,T-\sigma]$, we can find an integer $K \geq 0$ such that
\begin{equation*}
\mu_{K-1} \leq t < \mu_{K}\quad \text {and} \quad \left\|\mathbf{U}^{(K)}_{h,\tau}\left(\cdot, \lambda^{-2K} \left(t - \mu_{K-1}\right) \right)\right\|_\infty  < M,
\end{equation*}
where $\mu_q :=\sum_{i=0}^{q}\lambda^{2i}\tau^*_i$. Then, $\mathbf{U}_{h,\tau}(x,t)$ is defined as follows:
\begin{equation}\label{equ:defU}\small{
\mathbf{U}_{h,\tau}(x,t) = \left\{ 
\begin{array}{ll}
\lambda^{-\frac{2K}{p-1}}\mathbf{U}_{h,\tau}^{(K)}\left(\lambda^{-K}x, \lambda^{-2K} \left(t - \mu_{K-1}\right)\right) &\text{if $x\in \Omega_K$,}\\

\lambda^{-\frac{2(K-1)}{p-1}}\mathbf{U}_{h,\tau}^{(K-1)}\left(\lambda^{-(K-1)}x, \lambda^{-2(K-1)} \left(t - \mu_{K-2}\right)\right) &\text{if $x\in \Omega_{K-1}\backslash \Omega_K$,}  \\

\vdots & \\

\lambda^{-\frac{2}{p-1}}\mathbf{U}^{(1)}_{h,\tau}\left(\lambda^{-1}x, \lambda^{-2} \left(t - \mu_0\right)\right) &\text{if $x\in \Omega_{1}\backslash \Omega_2$,} \\
\mathbf{U}_{h,\tau}^{(0)}(x,t) &\text{if $x\in \Omega\backslash \Omega_1$.}
\end{array}
\right.}
\end{equation}
where $\Omega_k = (\lambda^k\xi_{k-1,i_{k-1}^-}, \lambda^k\xi_{k-1, i_{k-1}^+})$ for $k \geq 1$ and $\mathbf{U}^{(k)}_{h,\tau}$ is the linear interpolation defined in \eqref{equ:intST}.\\
One can see that the solution defined in \eqref{equ:defU} tends to infinity when $k$ goes to infinity. We say that the solution defined in \eqref{equ:defU} blows up in a finite time if 
\begin{equation}\label{equ:defiT}
T_{h,\tau} = \lim_{K\to +\infty}\sum_{k=0}^{K}\lambda^{2k}\tau_k^* < +\infty.
\end{equation}
The time $T_{h, \tau}$ is call the numerical blow-up time.
\begin{rema} \label{rema:blu}
We can see that $T_{h,\tau}$ defined in \eqref{equ:defiT} is finite if the solution $\mathbf{U}_{h,\tau}^{(k)}$ (defined from $\mathbf{U}_{n}^{(k)}$ by \eqref{equ:intST}) reaches the given threshold $M$ in a bounded number of time steps,  namely when $\bar{\tau} = \sup_{k \geq 0} \tau_k^* < +\infty$. In this case, we see that 
$$T_{h,\tau} \leq  \lim_{K\to +\infty} \bar{\tau}\sum_{k=0}^{K}\lambda^{2k} = \frac{\bar{\tau}}{1 - \lambda^2} < +\infty.$$
\end{rema}

\section{Convergence of the rescaling method} \label{sec:4}
This section is devoted to the convergence analysis of the rescaling method for problem \eqref{equ:semiHeatEqu} with $\beta \in \mathbb{R}$ and $q \in [1, 2)$ not necessarily $q = \frac{2p}{p+1}$, under some regularity assumptions. Note that the discrete problem \eqref{equ:sys} when $\beta = 0$ has already been treated in \cite{NBams11}. When $\beta \ne 0$, proceeding as for $\beta = 0$, the crucial step is to obtain a comparison principle for the discrete problem (see Lemma \ref{lemm:maxprin} below). Note that we could not prove analogous results for the equation \eqref{equ:GL}, since we already have no comparison principle  in the continuous case.
\begin{theo} 
\label{theo:conver}
Consider $h > 0$ sufficiently small and $\tau > 0$ such that $\tau \leq \frac{h^2}{2}$.  Let $\sigma > 0$, suppose that the problem \eqref{equ:semiHeatEqu} (with $q \in[1, 2)$) has a non-negative solution $u(x,t) \in \mathcal{C}^{4,2}\left([-1,1]\times [0,T-\sigma]\right)$ and  the initial data of \eqref{equ:sys} satisfies
\begin{equation*}
\sup_{x \in [-1,1]}|\phi_h(x)-u(x,0)|= \mathcal{O}(h^2) \quad as \quad h \rightarrow 0.
\end{equation*}
\noindent Then the solution $\mathbf{U}_{h,\tau}$ defined in \eqref{equ:defU} satisfies
\begin{equation*}
\sup_{(x,t)\in [-1,1]\times [0, T-\sigma]} |\mathbf{U}_{h,\tau}(x,t)- u(x,t)| = \mathcal{O}\left(h^2\right)  \quad as \quad h \rightarrow 0.
\end{equation*}
\end{theo}
\begin{rema}\label{rem:thpr}
The convergence of the rescaling method stated in Theorem \ref{theo:conver} is proved by a recursive application of Proposition \ref{prop:1} below. Therefore, it is enough to give the proof of this proposition.
\end{rema}
\begin{rema} The choice of the central difference approximation for the gradient in \eqref{equ:cendiff} is crucial to get the $\mathcal{O}(h^2)$ convergence in Theorem \ref{theo:conver}.
\end{rema}
\begin{rema} When $\beta \ne 0$ and $q < 2$, we have been unable to show that equation \eqref{equ:semiHeatEqu}  has a $\mathcal{C}^{4,2}$ solution (this is the case when $q = \frac{2p}{p+1}$). On the contrary, when $\beta = 0$ or $q \geq 2$, we do have $\mathcal{C}^{4,2}$ solutions (just take $p \geq 2$ and $u_0 \in \mathcal{C}^4([-1, 1])$, see \ref{ap:A} for a justification of this fact). Hence, our convergence result (Theorem \ref{theo:conver}) is at least meaningful when $\beta = 0$.
\end{rema}
\noindent One can see from the definition of $\mathbf{U}_{h,\tau}$ in \eqref{equ:defU} that $\mathbf{U}_{h,\tau}$ is constructed from $\mathbf{U}^{(k)}_{h,\tau}$ which is the solutions of the problem \eqref{equ:sysk}. It is reasonable then to consider the following problem with the non-zero Dirichlet condition,
\begin{equation}\label{equ:sh1}
\left\{ \begin{array}{rcll}
v_t(x,t) &=& v_{xx}(x,t) + g(v(x,t), v_x(x,t))&\;(x,t)\in (-L,L)\times(0,T),\\
v(-L, t) &=& v(L,t) = v_1(t)  &\quad t \in (0,T),\\
v(x,0)&=& v_0(x) &\quad x \in (-L,L),
\end{array}
\right.
\end{equation}
where $v(t): x \in (-L,L) \mapsto v(x,t) \in \mathbb{R}$, $p> 1$,
$$g(v,v_x) = |v|^{p-1}v + \beta |v_x|^q, \quad \text{with}\quad q = \frac{2p}{p+1}.$$

Let $I > 0$ and consider the grid $x_i = ih$, $-I\leq i \leq I$ where $h = \frac{L}{I}$. Let $\tau > 0$ be a time step and denote $t_n = n\tau$. Let 
$$\mathbf{V}_n= \left( V_{-I,n}, \dots, V_{0,n}, \dots, V_{I,n}\right)^T$$
be the approximation of $v(t_n)$ at grid points. Then, $\mathbf{V}_{n +1}$ is a solution of the following equation: for all $n \geq 0$, $i = -I + 1, \dots, I - 1$,
\begin{equation}\label{equ:sysV}
\left\{ \begin{array}{rcl}
\delta_t V_{i,n} &=& \delta^2_x V_{i,n} + g(V_{i,n}, \delta_x V_{i,n}) \\
V_{-I, n} &=& V_{I,n} = \psi_n, \quad V_{i,0} = \phi_i,
\end{array}
\right.
\end{equation}
where $\psi_n$ and $\phi_i$ stand for $\psi_n^{(k)}$ and $\phi_i^{(k)}$ introduced in \eqref{equ:phi1}, \eqref{equ:phik1}, \eqref{equ:psi1} and \eqref{equ:psik1}.\\

\noindent Let $\mathbf{V}_{h,n}(x)$ be the piecewise linear interpolation generated from $\mathbf{V}_{n}$ by \eqref{equ:intST}, then, we get the following results:

\begin{prop}\label{prop:1} Consider $h > 0$ sufficiently small and $\tau > 0$ such that $\tau \leq \frac{h^2}{2}$. Let $\eta \in (0,T)$, suppose that the problem \eqref{equ:sh1} (with $q \in [1, 2)$) has a non-negative solution $v \in \mathcal{C}^{4,2}([-L,L]\times [0,T-\eta])$, the initial data and boundary data  of \eqref{equ:sysV} satisfy 
\begin{eqnarray*}
\epsilon_1 = \sup_{x \in [-L,L]}|v(x,0) - \phi_h(x)| = o(h) \quad as \quad h \rightarrow 0, \\
\epsilon_2 = \sup_{t \in [0,T - \eta]}|v(L,t) - \psi_\tau(t)\| = o(1) \quad as \quad \tau \rightarrow 0,
\end{eqnarray*}
where $\phi_h$ and $\psi_\tau$ are the interpolations of $\phi_i$ and $\psi_n$ defined in \eqref{equ:intST}.
Then,
\begin{equation*}
\max_{0 \leq n \leq N }\left\|\mathbf{V}_{h,n} - v(t_n) \right\|_{\infty} = \mathcal{O}(\epsilon_1 + \epsilon_2 + h^2) \quad as \quad h \rightarrow 0,
\end{equation*}
where $N > 0$ is such that $t_N = N\tau \leq T-\eta$.
\end{prop}

\noindent We now state some properties of the discrete scheme  \eqref{equ:sysV}.
\begin{lemm}\label{lemm:sym}
Let $n = 1, 2, \dots, N$, $\mathbf{V}_n$ be the solution of \eqref{equ:sysV} and $\mathbf{V}_0$ be a symmetric data. Then, $\mathbf{V}_n$ is also symmetric for all $n = 0, 1, \dots N$.
\end{lemm}
\begin{proof}
It is straightforward from the symmetry of the data and the equation.
\end{proof}
\begin{rema}\label{rem:sym}
We can consider the problem \eqref{equ:sysV} on the half interval $[0, L]$ from now on. In particular, we have for $n \geq0$ and $i = 1, 2, \dots I-1$,
\begin{align}
\delta_x^2 V_{0,n}& = \frac{2V_{1,n} - 2V_{0,n}}{h^2}, \quad \delta_x^2 V_{i,n} = \frac{V_{i-1,n} - 2V_{i,n} + V_{i+1,n}}{h^2},\nonumber\\
\delta_x V_{0,n} &= 0, \quad \delta_x V_{i,n} = \frac{V_{i+1,n} - V_{i-1,n}}{2h}.\label{equ:deriAt0}
\end{align}
\end{rema}
\begin{rema} The convergence stated in Proposition \ref{prop:1} holds without the symmetric property.  However, we handle only symmetric data to simplify the proofs below.
\end{rema}

\begin{lemm}[\textbf{Positivity of the discrete solution}] \label{lem:posSo}
Let $n = 1, 2, \dots, N$ and $\mathbf{V}_n$ be the solution of \eqref{equ:sysV}. Suppose that $\mathbf{V}_0 \geq 0$ and $V_{I,n} \geq 0$ for $n = 0, 1, \dots, N$. Assume in addition that $\tau \leq \frac{h^2}{2}$ and $h \leq \left(\frac{2^q}{|\beta|M_0^{q - 1}} \right)^\frac{1}{2 - q}$ if $\beta < 0$, where $M_0 = \max \limits_{0 \leq n \leq N} \|\mathbf{V}_{n}\|_\infty$. Then, $\mathbf{V}_{n} \geq 0$ for all $n$ between $0$ and $N$. 
\end{lemm}
\begin{proof} By induction, we assume that $\mathbf{V}_k \geq 0$ for all $k = 0, 1, \dots, n$. We need to show that $\mathbf{V}_{n+1} \geq 0$. Using \eqref{equ:sysV}, we see that 
$$
V_{0,n+1} = \left(1 - \frac{2\tau}{h^2}\right)V_{0,n} + \frac{2\tau}{h^2}V_{1,n} + \tau V_{0,n}^p,
$$
where we used the fact that $\delta_x V_{0,n} = 0$ from Remark \ref{rem:sym}. From the restriction $\tau \leq \frac{h^2}{2}$, we have $V_{0,n+1} \geq 0$.\\
For $i$ between $1$ and $I-1$, we have
$$
V_{i,n+1} = \left(1 - \frac{2\tau}{h^2}\right)V_{i,n} + \frac{\tau}{h^2}\left(V_{i+1,n} + V_{i-1,n}\right) + \tau V_{i,n}^p + \frac{\tau\beta}{(2h)^q}\left| V_{i+1,n} - V_{i-1,n}\right|^q.
$$
If $\beta \geq 0$ and $\tau \leq \frac{h^2}{2}$, we directly infer the desired result. If $\beta < 0$, we have for $i = 1, \dots, I-1$,
\begin{align*}
 V_{i,n+1} &\geq \left(1 - \frac{2\tau}{h^2}\right)V_{i,n} + \frac{\tau}{h^2}\left(V_{i+1,n} + V_{i-1,n}\right) + \tau V_{i,n}^p - \frac{\tau|\beta|}{(2h)^q}\left(V_{i+1,n}^q + V_{i-1,n}^q\right)\\
 & = \left(1 - \frac{2\tau}{h^2}\right)V_{i,n} + \tau V_{i,n}^p + \frac{\tau}{h^q}\left(\frac{1}{h^{2 - q}} - \frac{|\beta|}{2^q}V_{i+1,n}^{q - 1} \right)V_{i+ 1,n}\\
 &\qquad \qquad \qquad \qquad + \frac{\tau}{h^q}\left(\frac{1}{h^{2 - q}} - \frac{|\beta|}{2^q}V_{i-1,n}^{q - 1} \right)V_{i- 1,n}. 
\end{align*}
Here we used the induction assumption that $V_{i,n} \geq 0$ for $i = 0, 1, \dots, I$. To obtain $V_{i,n+1} > 0$, then it requires the following restrictions
$$ \frac{\tau}{h^2} \leq \frac{1}{2} \quad \text{and} \quad \frac{1}{h^{2 - q}} - \frac{|\beta|}{2^q}M_0^{q-1} \geq 0.$$
Recall that $q \in [1,2)$, then the last condition yields $h \leq \left(\frac{2^q}{|\beta|M_0^{q - 1}} \right)^\frac{1}{2 - q}$. This ends the proof of Lemma \ref{lem:posSo}.
\end{proof}

\noindent The following lemma is a discrete version of the maximum principle.
\begin{lemm} \label{lemm:maxprin}
Let $\mathbf{b}_n = (b_{0,n}, b_{1,n}, \dots, b_{I,n})^T, \mathbf{c}_n = (c_{0,n}, c_{1,n}, \dots, c_{I,n})^T$ be two vectors such that $\mathbf{b}_n \geq 0$ and $\mathbf{c}_n$ is bounded. Let $\mathbf{V}_n = (V_{0,n}, V_{1,n}, \dots, V_{I,n})^T$ satisfy
\begin{eqnarray*} 
\delta_t V_{i,n} - \delta^2_x V_{i,n} - b_{i,n}V_{i,n} - c_{i,n}\delta_xV_{i,n} &\geq& 0, \quad 0\leq i\leq I-1,\\
V_{i,0} &\geq& 0, \quad 0\leq i\leq I,\\
V_{I,n} &\geq& 0, \quad n \geq 0.
\end{eqnarray*}
If $\tau \leq \frac{h^2}{2}$ and $h \leq \frac{2}{\|\mathbf{c}_n\|_{\infty}}$, then $\mathbf{V}_n \geq 0$ for all $n \geq 0$ .
\end{lemm}

\begin{rema}
Note that as before, we handle symmetric data in this lemma. That is the reason why we focus only on $i \geq 0$. Note also that \eqref{equ:deriAt0} is useful for this lemma.
\end{rema}
\begin{proof}[\textbf{Proof of Lemma \ref{lemm:maxprin}}] We proof this lemma by induction. Assume that $\mathbf{V}_{k} \geq 0$ for $k = 0, 1, \dots, n$. Let us show that $\mathbf{V}_{n+1} \geq 0$. A straightforward calculation yields 
$$V_{0,n+1} \geq \frac{2\tau}{h^2} V_{1,n} + \left(1 - \frac{2\tau}{h^2}  \right)V_{0,n} + \tau b_{0,n}V_{0,n},$$
for $i$ between $1$ and $I-1$, we have
\begin{align*}
V_{i,n+1} &\geq \left(\frac{\tau}{h^2} - \frac{\tau}{2h}c_{i,n}\right) V_{i-1,n} + \left(1 - \frac{2\tau}{h^2}  \right)V_{i,n}\\
&\quad + \tau b_{i,n}V_{i,n} + \left(\frac{\tau}{h^2} + \frac{\tau}{2h}c_{i,n}\right) V_{i+1,n}.
\end{align*}
\noindent Since $\tau \leq \frac{h^2}{2}$, $h \leq \frac{2}{\|\mathbf{c}_n\|_\infty}$ and $\mathbf{b}_{n}, \mathbf{V}_{n}$ are non-negative, we deduce that $\mathbf{V}_{n+1} \geq 0$. This ends the proof.
\end{proof}
Let us now give the proof of Proposition \ref{prop:1}.
\begin{proof}[\textbf{Proof of Proposition \ref{prop:1}}] Under the hypothesis stated in Proposition \ref{prop:1}, we see that if $h$ is small enough, we may consider $K \leq N$ be the greatest value such that for all $n < K$,
\begin{equation}\label{equ:conVKN}
\max_{0 \leq i \leq I} |V_{i,n} - v(x_i,t_n) | < 1 \; \text{and} \; \max_{0 \leq i \leq I-1} \big||\delta_x V_{i,n}| - |\delta_x v(x_i,t_n)| \big| < 1.
\end{equation}
From the the fact that $v_0 \geq 0$ and Lemma \ref{lem:posSo}, we see that the solution of \eqref{equ:sysV} is non-negative. Furthermore, since $v(x,t) \in C^{4,2}\left([-L,L]\times [0, T-\eta]\right)$, there exist positive constants $C_1$ and $C_2$ such that for all $(x,t) \in [-L,L]\times [0, T-\eta]$,
$$|v_x^{(i)}(x,t)| \leq C_1, \quad 0 \leq i \leq 4 \quad \text{and} \quad |v_t^{(j)}(x,t)| \leq C_2, \quad 0 \leq j \leq 2.$$
Thus, we obtain from the triangle inequality that 
$$\max_{0 \leq i \leq I} |V_{i,n}| \leq 1 + C_1 \quad \text{and} \max_{0 \leq i \leq I-1} |\delta_x V_{i,n}| \leq 1 + C_2, \quad \text{for } \quad n < K. $$
Using Taylor's expansion and \eqref{equ:sh1}, we derive for all $1 \leq i \leq I-1, 0 < n < K$,
\begin{equation*}\label{equ:tay1}
\delta_t v(x_i, t_n) \leq \delta_x^2 v(x_i, t_n) + v^p(x_i, t_n) + \beta |\delta_x v(x_i,t_n)|^q+ C_3h^2 + C_4\tau,
\end{equation*}
where $C_3, C_4$ are positive constants.\\

\noindent Let $e_{i,n} = V_{i,n} - v(x_i, t_n)$ be the discretization error. We have, 
\begin{equation*}\label{equ:rea3}
\delta_t e_{i,n} \leq \delta_x^2 e_{i,n} + V_{i,n}^p - v^p(x_i,t_n) + \beta (|\delta_x V_{i,n}|^q - |\delta_x v(x_i,t_n)|^q) +  C_3h^2 + C_4\tau.
\end{equation*}
Applying the mean value theorem, we get
\begin{equation*}
\delta_t e_{i,n} \leq \delta_x^2 e_{i,n} + p\xi_{i,n}^{p-1}e_{i,n} + \beta q |\theta^{q-2}_{i,n}|\theta_{i,n}\delta_x e_{i,n} + C_3h^2 + C_4\tau,
\end{equation*}
where $\xi_{i,n}$ is an intermediate value between $V_{i,n}$ and $v(x_i, t_n)$, $\theta_{i,n}$ is between $\delta_x V_{i,n}$ and $\delta_{x}v(x_i,t_n)$.  \\
Since $\tau \leq \frac{h^2}{2}$, we then obtain for all $i\leq I-1$ and $n < K$,
\begin{equation}\label{equ:reaEr1}
\delta_t e_{i,n} \leq \delta_x^2 e_{i,n} + p\xi_{i,n}^{p-1} e_{i,n} +\beta q |\theta^{q-2}_{i,n}|\theta_{i,n}\delta_x e_{i,n} + C_5h^2,
\end{equation}
where $C_5 = C_3 + C_4/2$.\\
We now consider the function 
$$z(x,t) = e^{At + x^2} \left( \epsilon_1 + \epsilon_2 + Qh^2 \right),$$
where $A, Q$ are positive constants which will be chosen later.\\
We observe that, for $0 \leq i \leq I$,
\begin{align*} 
z(x_i,0) &= e^{x_i^2}  \left(\epsilon_1 + \epsilon_2 + Qh^2\right)
\geq e_{i,0}, \\ 
z(x_{I}, t_n) &= e^{At_n + x_I^2}\left( \epsilon_1 + \epsilon_2 + Q h^2\right) \geq e_{I,n},
\end{align*}
and 
\begin{align*}
z_t(x,t) - z_{xx}(x,t) & - p\xi_{i,n}^{p-1} z(x,t) - \beta q |\theta^{q-2}_{i,n}|\theta_{i,n} z_x(x,t)\\
&=(A - 2 - p\xi_{i,n}^{p-1} - 4x^2 - 2 \beta q |\theta^{q-2}_{i,n}|\theta_{i,n} x)z(x,t).
\end{align*}
Using Taylor's expansion, we get
\begin{align*}
\delta_t z(x_i,t_n) &- \delta_x^2 z(x_i,t_n) - p\xi_{i,n}^{p-1} z(x_i,t_n) - \beta q |\theta^{q-2}_{i,n}|\theta_{i,n} \delta_x z(x_i,t_n)\\
&= (A - 2 - p\xi_{i,n}^{p-1} - 4x_i^2 - 2\beta q |\theta^{q-2}_{i,n}|\theta_{i,n} x_i)z(x_i,t_n) \\
& \qquad + \frac{h^2}{12}z_{xxxx}(\tilde{x}_i,t_n) + \beta q |\theta^{q-2}_{i,n}|\theta_{i,n}\frac{h^2}{6}z_{xxx}(\bar{x}_i, t_n)- \frac{\tau}{2}z_{tt}(x_i,\tilde{t}_n),
\end{align*}
where $\tilde{t}_n \in [t_n, t_{n+1}]$ and $\tilde{x}_i, \bar{x}_i \in [x_{i-1}, x_{i+1}]$.\\
By taking $A, Q$ large enough, then $h$ small enough such that the right-hand side of the above equation is lager than $C_5 h^2$, we obtain
\begin{equation}\label{equ:reZZ}
\delta_t z(x_i,t_n) - \delta_x^2 z(x_i,t_n)  - p\xi_{i,n}^{p-1} z(x_i,t_n) - \beta q |\theta^{q-2}_{i,n}|\theta_{i,n} \delta_x z(x_i,t_n) \geq C_5 h^2.
\end{equation}
From \eqref{equ:reaEr1} and \eqref{equ:reZZ}, applying Lemma \ref{lemm:maxprin} to $z(x_i,t_n) - e_{i,n}$ with $b_{i,n} = p\xi_{i,n}^{p-1} \geq 0$ and $c_{i,n} = \beta q |\theta^{q-2}_{i,n}|\theta_{i,n}$ bounded, we get $e_{i,n} \leq z(x_i,t_n)$ for $0\leq i \leq I$ and $0 \leq n < K$. By the same way, we also show that $-e_{i,n} \leq z(x_i,t_n)$ for $0 \leq i \leq I$ and $0 \leq n \leq K$. In conclusion, we derive
\begin{equation*}
\max_{0 \leq i \leq I} |V_{i,n} - v(x_i, t_n)| \leq z(x_i,t_n) \leq e^{AT + L^2}(\epsilon_1 + \epsilon_2 + Qh^2), \quad\text{for} \quad  n < K.
\end{equation*}

\noindent Let us show that $K = N$. Assuming by contradiction that $K < N$, we have
\begin{equation*}
1 \leq \max_{0\leq i \leq I}|V_{i,K} - v(x_i, t_K)| \leq z(x_i,t_K) \leq e^{AT + L^2}(\epsilon_1 + \epsilon_2 + Qh^2).
\end{equation*}
But this contradicts with the fact that the last term in the above inequality tends to zero as $h$ tends to zero. This concludes the proof of Proposition \ref{prop:1}. Since Theorem \ref{theo:conver} is a consequence of Proposition \ref{prop:1}, as we pointed in Remark \ref{rem:thpr}, this is also the conclusion of the proof of Theorem \ref{theo:conver}.
\end{proof}
\section{Numerical results} \label{sec:6}
The numerical experiments presented in this section are performed with the initial data 
\begin{equation}\label{equ:dataEx}
u_0(x) = A\left(1 + \cos(\pi x)\right),\quad x \in (-1,1),
\end{equation}
where $A = 1.2$. For the non-linearity power, we take $p =5$ and $p=7$. Let us recall from Remark \ref{rema:56} that the threshold $M$ is given by $M= \lambda^{-\frac{2}{p-1}}\|u_0\|_\infty$. Therefore, the parameters of the algorithm are $\lambda = \frac{1}{2}$,  $\alpha = 0.4$, $M = 2A \times 2^{1/2}$ if $p =5$ and $M = 2A \times 2^{1/3}$ if $p=7$. For the time step $\tau$, we take $\tau = \frac{h^2}{4}$ where $h = \frac{2}{I}$. We perform the experiments with $I = 50, 100, 160, 250, 320, 400$.

\subsection{The semilinear heat equation ($\beta =0$ in equation \eqref{equ:semiHeatEqu}).} \label{sec:anaSHE}
Note that the original paper of Berger and Kohn \cite{BKcpam88} was totally devoted to this case. We now recall the  assertion that the value $\tau_k^*$ is independent of $k$ and tends to a constant as $k$ tends to infinity. In order to  establish this assertion, we recall from Merle and Zaag \cite{MZcpam98} that 
\begin{equation}\label{equ:limUtT}
\lim_{t \to T} (T - t)^{\frac{1}{p-1}}\|u(t)\|_\infty = \kappa, \quad \text{with}\quad \kappa = (p-1)^{-\frac{1}{p-1}}.
\end{equation}
Then, using \eqref{equ:u_kp1} we see that
\begin{equation}\label{equ:relau_uk}
u^{(k)}(\xi_k,\tau_k^*) = \lambda^{\frac{2}{p-1}}u^{(k-1)}(\lambda \xi_k, \tau^*_{k-1} + \lambda^2 \tau_k^*)= \dots = \lambda^\frac{2k}{p-1}u(\lambda^k \xi_k, t_k),
\end{equation}
where $t_k = \tau_0^* + \lambda^2\tau_1^* + \dots + \lambda^{2k}\tau_k^*.$\\
Hence, it holds that 
$$(T - t_k)^\frac{1}{p-1}\|u(t_k)\|_\infty = (T - t_k)^\frac{1}{p-1} \lambda^\frac{-2k}{p-1}\|u^{(k)}(\tau_k^*)\|_\infty.$$
Since $\|u^{(k)}(\tau_k^*)\|_\infty = M$, we obtain 
\begin{equation}\label{equ:T_tk}
T -t_k = \lambda^{2k}M^{1-p}(p-1)^{-1} + o(1) \quad \text{as} \quad k \to \infty
\end{equation}
on the one hand.\\

On the other hand, we get 
\begin{align*}
\tau_k^* & = \lambda^{-2k}(t_k - t_{k-1}) = \lambda^{-2k}\left((T - t_{k-1}) - (T - t_k) \right)\\
& = M^{1-p}(p-1)^{-1}(\lambda^{-2} - 1) + o(1).
\end{align*}
Consequently, we obtain 
\begin{equation}\label{equ:limintoftauk}
\lim_{k \to +\infty} \tau_k^* = M^{1-p}(p-1)^{-1}(\lambda^{-2} - 1).
\end{equation}

Figure \ref{fig:1p5} presents the computed values of $\tau_k^*$ when $p=5$, for different values of $I$. The values of $\tau_k^*$  are tabulated in Tables \ref{tab:1p5} and \ref{tab:1p7} for some selected values of $k$. These experimental results are in agreement with the fact that $\tau_k^*$ tends to the constant indicated in the right-hand side of \eqref{equ:limintoftauk} as  $k$ tends to infinity.
\begin{figure}[!htbp]
\begin{center}
\includegraphics[scale = 0.3]{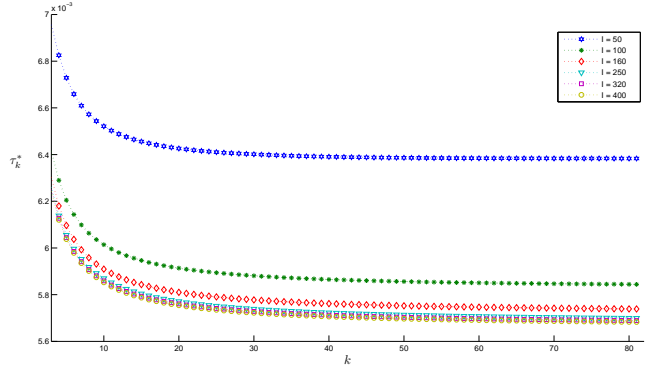}
\caption{The computed values of $\tau_k^*$ are plotted against $k$ when $p=5$.}
\label{fig:1p5}
\end{center}
\end{figure}
\begin{table}[!htbp]
\begin{center}
\tabcolsep = 3mm
\begin{tabular}{ccccccc}
\hline
$k$ & I = 50 & I = 100 & I = 160 & I = 250 & I = 320 & I = 400 \\
\hline
20 &0.6426 &0.5913 &0.5810 &0.5771 &0.5760 &0.5755 \\
30 &0.6401 &0.5881 &0.5778 &0.5739 &0.5728 &0.5722 \\
40 &0.6391 &0.5865 &0.5762 &0.5723 &0.5712 &0.5706 \\
50 &0.6386 &0.5856 &0.5752 &0.5713 &0.5703 &0.5697 \\
60 &0.6384 &0.5851 &0.5746 &0.5707 &0.5697 &0.5691 \\
70 &0.6384 &0.5847 &0.5742 &0.5703 &0.5693 &0.5687 \\
80 &0.6383 &0.5844 &0.5739 &0.5700 &0.5689 &0.5683 \\
\hline
\end{tabular}
\end{center} 
\caption{The computed values of $\tau_k^* (\times 10^{-2})$ when $p=5$.}
\label{tab:1p5}
\end{table}
\begin{table}[!htbp]
\begin{center}
\tabcolsep = 3mm
\begin{tabular}{ccccccc}
\hline
$k$ & I = 50 & I = 100 & I = 160 & I = 250 & I = 320 & I = 400 \\
\hline
20 &0.1279 &0.0826 &0.0726 &0.0688 &0.0677 &0.0671 \\
30 &0.1279 &0.0825 &0.0724 &0.0686 &0.0675 &0.0670 \\
40 &0.1279 &0.0825 &0.0724 &0.0685 &0.0674 &0.0669 \\
50 &0.1279 &0.0825 &0.0723 &0.0684 &0.0674 &0.0668 \\
60 &0.1279 &0.0825 &0.0723 &0.0684 &0.0673 &0.0667 \\
70 &0.1279 &0.0825 &0.0723 &0.0683 &0.0673 &0.0667 \\
80 &0.1279 &0.0825 &0.0723 &0.0683 &0.0673 &0.0667 \\
\hline
\end{tabular}
\end{center} 
\caption{The computed values of $\tau_k^*(\times 10^{-2})$ when $p=7$.}
\label{tab:1p7}
\end{table}
In Figure \ref{fig:2blrate}, we show the plot of $\left\|\mathbf{U}_{h,\tau}(t)\right\|_{\infty}$ versus $(T_{h,\tau} - t)$ in log-scale where $T_{h,\tau}$ is given by $T_{h,\tau} = \sum_{k = 0}^K \lambda^{2k}\tau_k^*$.
The slope of the obtained curves measures the blow-up rate. As expected from \eqref{equ:limUtT}, these slopes for $p = 5$ and $p=7$ are $\frac{1}{4}$ and $\frac{1}{6}$ respectively.
\begin{figure}[!htbp]
\begin{center}
\includegraphics[scale = 0.3]{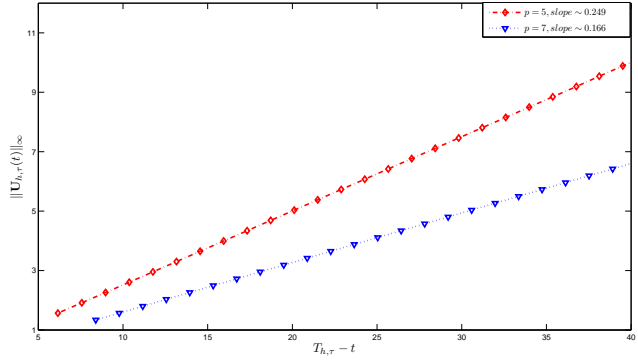}
\caption{Blow-up rate (in log-scale) when $p=5$ and $p=7$, for $I = 400$.}
\label{fig:2blrate}
\end{center}
\end{figure}

In order to examine the theoretical profile defined in \eqref{equ:defFshe}, we recall the method of Berger and Kohn \cite{BKcpam88} to consider the rescaled profile,
\begin{equation}\label{equ:resPro}
z \to u^{(k)}\left(z\lambda^{-1}\xi_{k-1}^+, \tau_k^*\right), \quad |z| < 1,
\end{equation}
where $\xi_{q}^+ = \xi_{q, i_{q}^+}$.
Using the \emph{semilarity variables} defined in \eqref{equ:simiWU} and \eqref{equ:relau_uk}, we get 
\begin{equation}
\label{equ:relaUk_Wk}
u^{(k)}(\xi_k, \tau_k^*) = \lambda^{\frac{2k}{p-1}}(T - t_k)^{-\frac{1}{p-1}}w\left(\lambda^k \frac{\xi_k}{\sqrt{T-t_k}}, s_k\right),
\end{equation}
where $s_k = -\log{(T-t_k)}$.\\
We recall from \eqref{equ:T_tk} that
\begin{equation}\label{equ:estT_tk}
T -t_k \sim \left(\lambda^{2k}M^{1-p}\right) (p-1)^{-1}.
\end{equation}
Substituting \eqref{equ:estT_tk} into \eqref{equ:relaUk_Wk} yields
\begin{equation*}
u^{(k)}(\xi_k, \tau_k^*) \sim (p-1)^{\frac{1}{p-1}} Mw\left(\sqrt{p-1}M^{\frac{p-1}{2}} \xi_k,s_k\right).
\end{equation*}
From \eqref{equ:preproSLHn}, we replace $\xi_k$ by $z\lambda^{-1}\xi_{k-1}^+$ to obtain 
\begin{equation}\label{{equ:rel4He}}
u^{(k)}(z\lambda^{-1}\xi_{k-1}^+, \tau_k^*) \sim (p-1)^{\frac{1}{p-1}} Mf\left(\sqrt{p-1}M^{\frac{p-1}{2}} z\lambda^{-1} \frac{\xi_{k-1}^+}{\sqrt{s_k}}\right).
\end{equation}
Assume that $\frac{\xi_{k-1}^+}{\sqrt{s_k}}$ tends to $\zeta$. Using the fact that $\alpha M = u^{(k-1)}(\xi_{k-1}^+, \tau_k^*)$ yields
\begin{equation*}
\alpha M = M(p-1)^\frac{1}{p-1} f\left(\sqrt{p-1}M^{\frac{p-1}{2}} \zeta \right),
\end{equation*}
or 
\begin{equation*}
\alpha  = (p-1)^\frac{1 }{p-1}f\left(A \zeta \right), \quad  A = \sqrt{p-1}M^{\frac{p-1}{2}}.
\end{equation*}
Using the definition of $f$ in \eqref{equ:defFshe}, it holds that 
\begin{equation*} 
\alpha =(p-1)^\frac{1 }{p-1}\left(p-1 + \frac{(p-1)^2}{4p} |A \zeta |^2 \right) ^\frac{-1}{p-1}.\nonumber\\
\end{equation*}
A straightforward computation gives
\begin{equation}\label{equ:Azeta}
|A\zeta|^2 = \frac{4p}{p-1}\left(\alpha^{1-p} -1 \right).
\end{equation}
Using \eqref{equ:Azeta} and \eqref{{equ:rel4He}}, we arrive at
\begin{eqnarray} 
u^{(k)}(z \lambda^{-1}y_{k-1}^+, \tau_k^*) &\sim & M(p-1)^\frac{1}{p-1} f\left( \lambda^{-1}z (A\zeta)\right)^{-\frac{1}{p-1}} \nonumber\\
&\sim& M (p-1)^\frac{1 }{p-1} \left(p-1 + \frac{(p-1)^2}{4p} \lambda^{-2} z^2 |A \zeta|^2 \right) ^{-\frac{1}{p-1}}\nonumber\\
&\sim& M \left(1 + \left(\alpha^{1-p}- 1\right)\lambda^{-2} z^2 \right)^{-\frac{1}{p-1}}. \label{equ:profnumSHE}
\end{eqnarray}
Since $\mathbf{U}^{(k)}_{h,\tau}$ converges to $u^{(k)}$ as $h$ goes to zero, it holds that  
\begin{equation}\label{equ:numProSHE}
\mathbf{U}^{(k)}_{h,\tau}(z \lambda^{-1}\xi_{k-1}^+, \tau_k^*) \sim M \left(1 + \left(\alpha^{1-p}- 1\right)\lambda^{-2} z^2 \right)^{-\frac{1}{p-1}}, \quad |z| < 1.
\end{equation}
We expect that the left-hand side of \eqref{equ:numProSHE} tends to the predicted profile as $k$ tends to infinity. Figures \ref{fig:3p5} and \ref{fig:3p7} display this relationship after 80 iterations with $I = 400$. Figures \ref{fig:4p5} and \ref{fig:4p7} illustrate the output of our algorithm using $I = 400$ at some selected values of $k$. As $k$ increases these computed profiles converge to the profile shown in Figures \ref{fig:3p5} and \ref{fig:3p7} respectively. We give in Tables \ref{tab:3} and \ref{tab:4} the error in $L^\infty$-norm between the computed profiles and the predicted profile using various values of $I$ in both cases $p = 5$ and $p =7$. The expression of the error is given by
$$e^{(k)}_{h,\tau} = \sup_{-1 \leq z \leq 1} \left| \mathbf{U}_{h,\tau}^{(k)}(z\lambda^{-1}\xi_{k-1}^+, \tau_k^*) - M[1 + (\alpha^{1-p} - 1)\lambda^{-2}z^2]^{-\frac{1}{p-1}}\right|.$$
The graphs of $e^{(k)}_{h,\tau}$ versus $h$ in log-scale are visualized in Figures \ref{fig:5p5} and \ref{fig:5p7}. We observe in those figures that the error tends to zeros as $h \to 0$. We note that the error $e^{(k)}_{h,\tau}$ includes two sources: the discretization error in using the scheme \eqref{equ:sys} and the asymptotic error which refers to the behavior of $w(y,s)$ as $s$ tends to infinity. 
\begin{figure}[!htbp]
\begin{center}
\includegraphics[scale = 0.3]{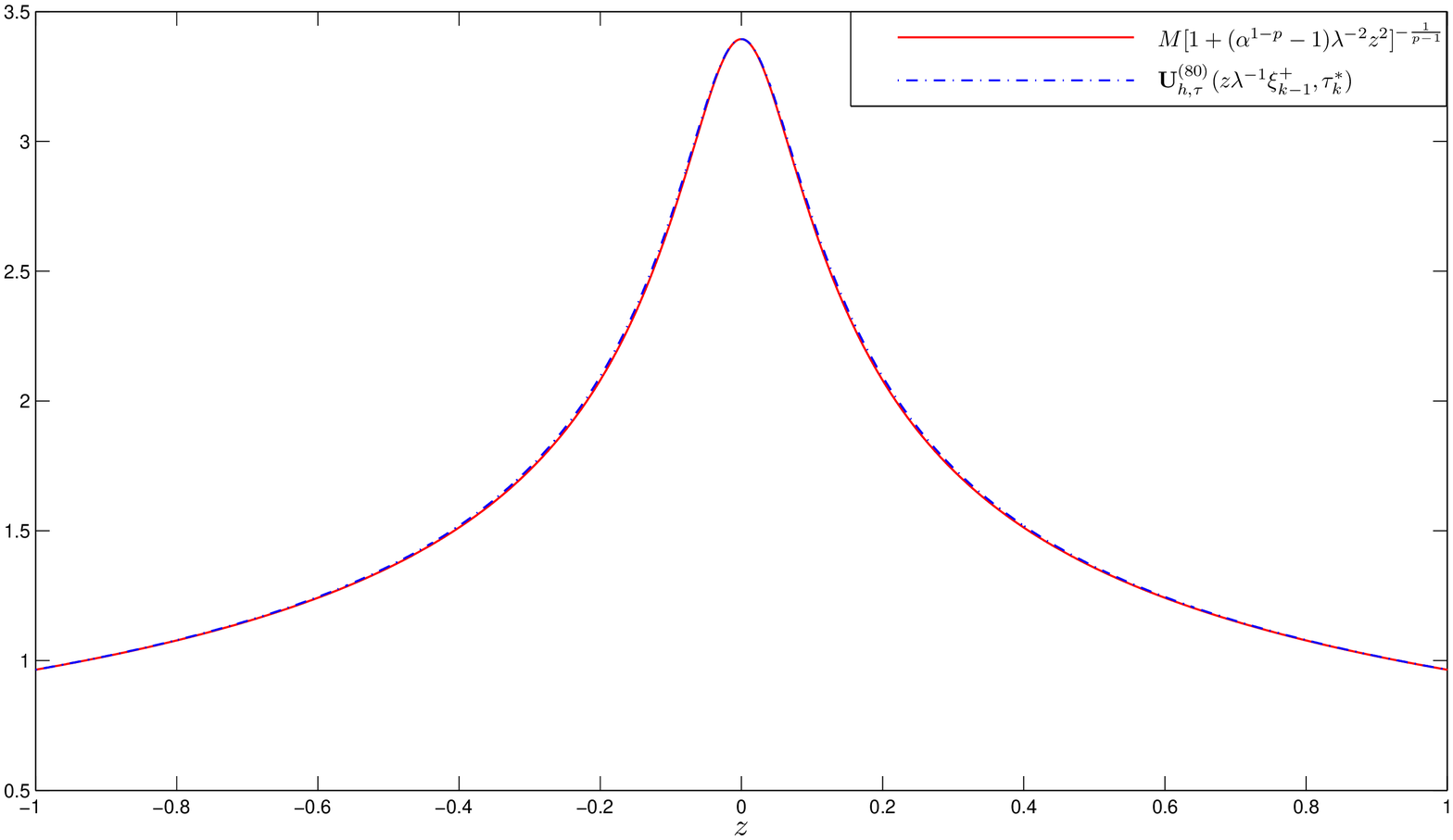}
\caption{The computed profile \eqref{equ:resPro} for $k = 80$ with $I = 400$ and the predicted profile \eqref{equ:numProSHE} with $p=5$.}
\label{fig:3p5}
\end{center}
\end{figure}
\begin{figure}[!htbp]
\begin{center}
\includegraphics[scale = 0.3]{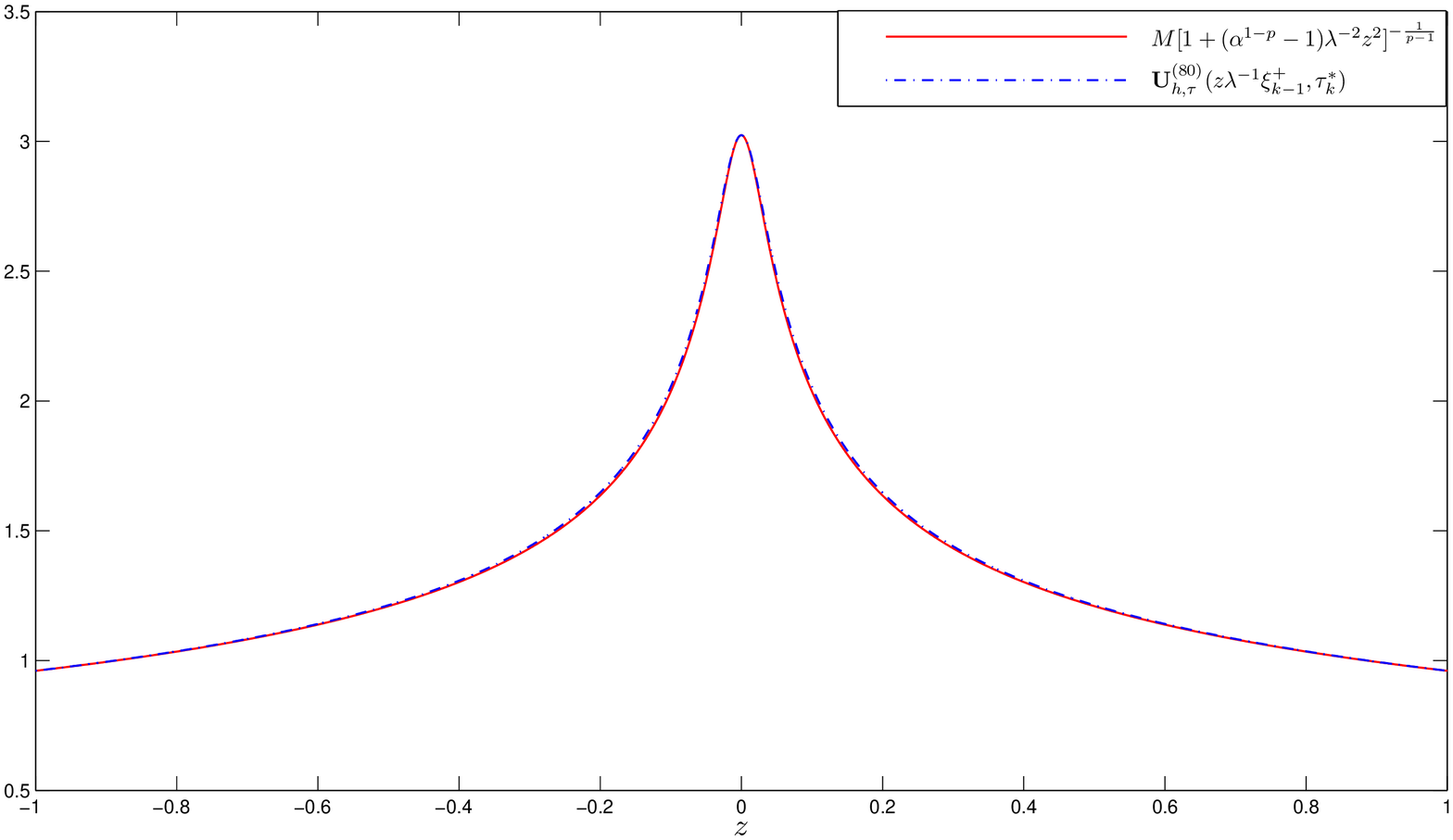}
\caption{The computed profile  \eqref{equ:resPro} for $k = 80$ with $I = 400$ and the predicted profile \eqref{equ:numProSHE} with $p=7$.}
\label{fig:3p7}
\end{center}
\end{figure}
\begin{figure}[!htbp]
\begin{center}
\includegraphics[scale = 0.3]{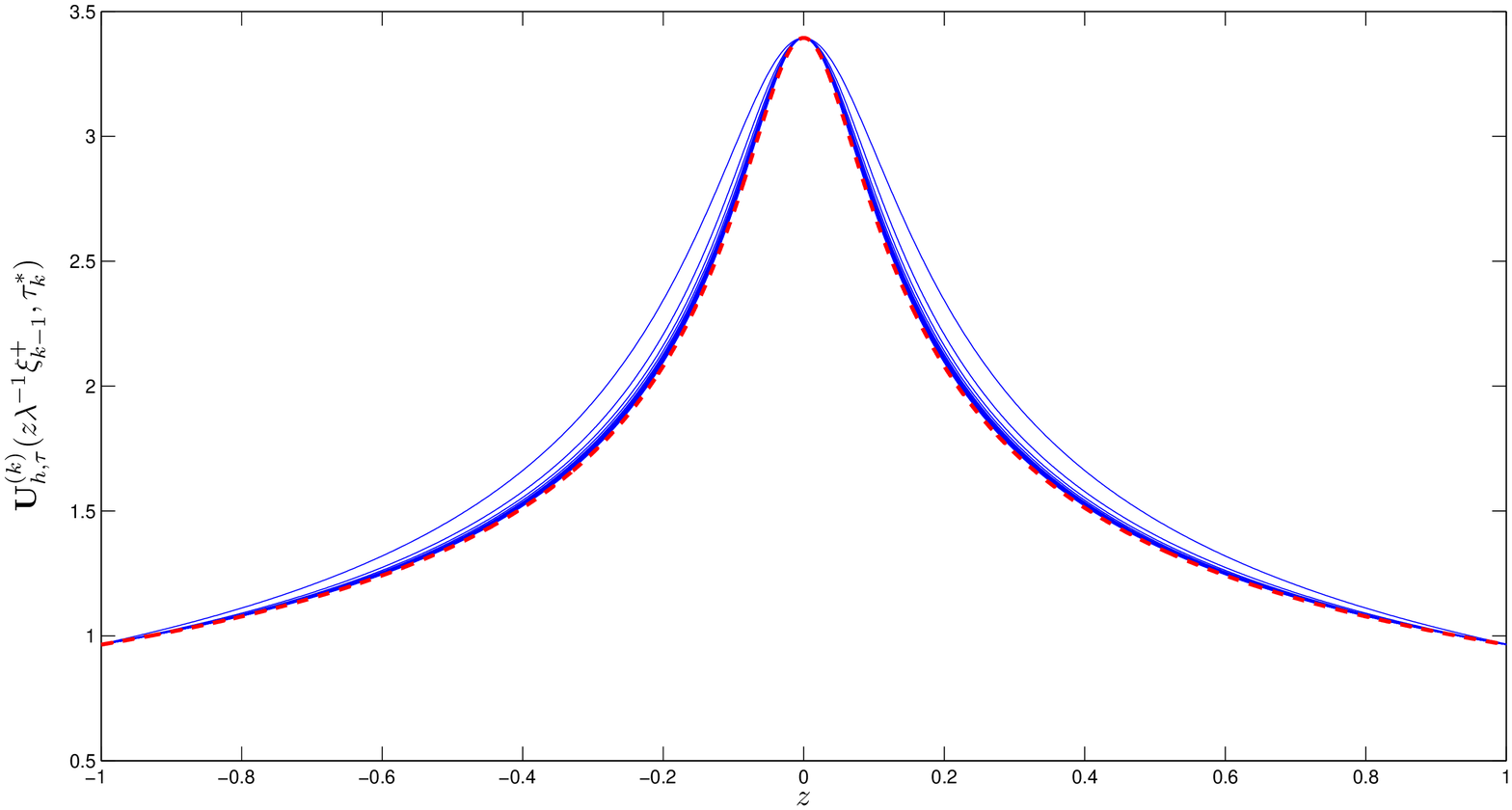}
\caption{The computed profiles as in \eqref{equ:resPro} for selected values of $k$ with $I = 400$ and $p = 5$.}
\label{fig:4p5}
\end{center}
\end{figure}
\begin{figure}[!htbp]
\begin{center}
\includegraphics[scale = 0.3]{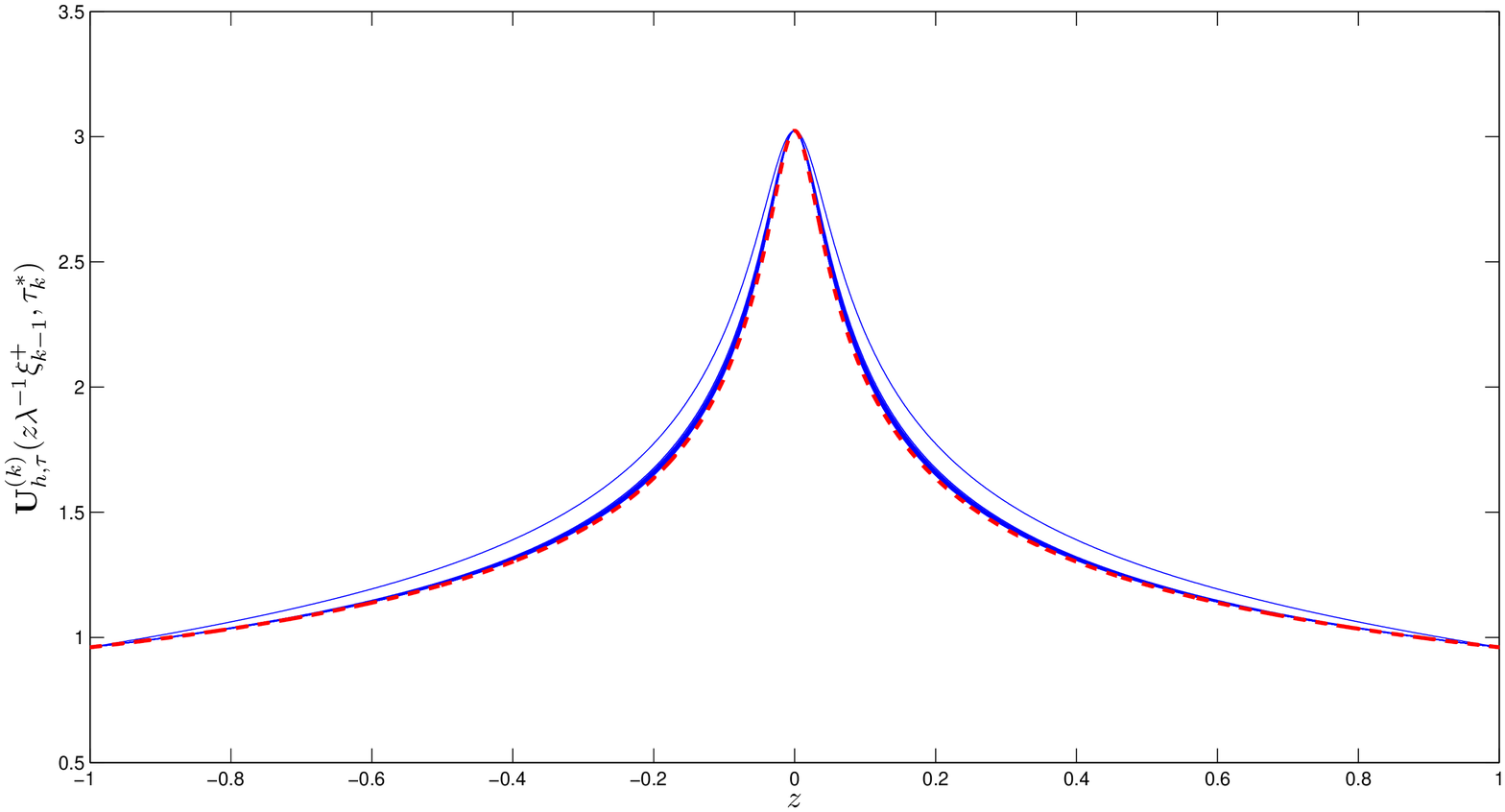}
\caption{The computed profiles as in \eqref{equ:resPro} for selected values of $k$ with $I = 400$ and $p = 7$.}
\label{fig:4p7}
\end{center}
\end{figure}
\begin{table}[!htbp]
\begin{center}
\tabcolsep = 3mm
\begin{tabular}{ccccccc}
\hline
$k$ & I = 50 & I = 100 & I = 160 & I = 250 & I = 320 & I = 400 \\
\hline
10&	0.2023 &  0.1581 &  0.1501  & 0.1466  & 0.1446  & 0.1432 \\
20& 0.1336 &  0.0858 &  0.0787  & 0.0735  & 0.0722  & 0.0715 \\
30& 0.1213 &  0.0636 &  0.0517  & 0.0480  & 0.0483  & 0.0474 \\
40& 0.1141 &  0.0504 &  0.0404  & 0.0376  & 0.0351  & 0.0354 \\
50& 0.1091 &  0.0444 &  0.0341  & 0.0297  & 0.0289  & 0.0276 \\
60& 0.1076 &  0.0409 &  0.0300  & 0.0249  & 0.0241  & 0.0231 \\
70& 0.1068 &  0.0372 &  0.0255  & 0.0214  & 0.0209  & 0.0202 \\
80& 0.1066 &  0.0354 &  0.0232  & 0.0188  & 0.0182  & 0.0174 \\
\hline
\end{tabular}
\end{center} 
\caption{Error in $L^\infty$-norm between the computed profile and the predicted profile for selected values of $k$ using various values of $I$ with $p=5$.}
\label{tab:3}
\end{table}
\begin{table}[!htbp]
\begin{center}\label{tab:3p7}
\tabcolsep = 3mm
\begin{tabular}{ccccccc}
\hline
$k$ & I = 50 & I = 100 & I = 160 & I = 250 & I = 320 & I = 400 \\
\hline
10&	0.2900  & 0.1930  & 0.1205 &  0.0887  & 0.0792  & 0.0730\\
20& 0.2748  & 0.1757  & 0.0970 &  0.0651  & 0.0556  & 0.0516\\
30& 0.2711  & 0.1715  & 0.0880 &  0.0545  & 0.0451  & 0.0398\\
40& 0.2725  & 0.1699  & 0.0843 &  0.0484  & 0.0391  & 0.0337\\
50& 0.2723  & 0.1696  & 0.0822 &  0.0441  & 0.0351  & 0.0295\\
60& 0.2706  & 0.1695  & 0.0810 &  0.0421  & 0.0322  & 0.0265\\
70& 0.2726  & 0.1694  & 0.0803 &  0.0403  & 0.0298  & 0.0240\\
80& 0.2720  & 0.1694  & 0.0802 &  0.0393  & 0.0285  & 0.0224\\
\hline
\end{tabular}
\end{center} 
\caption{Error in $L^\infty$-norm between the computed profile and the predicted profile for selected values of $k$ using various values of $I$ with $p=7$.}
\label{tab:4}
\end{table}
\begin{figure}[!htbp]
\begin{center}
\includegraphics[scale = 0.3]{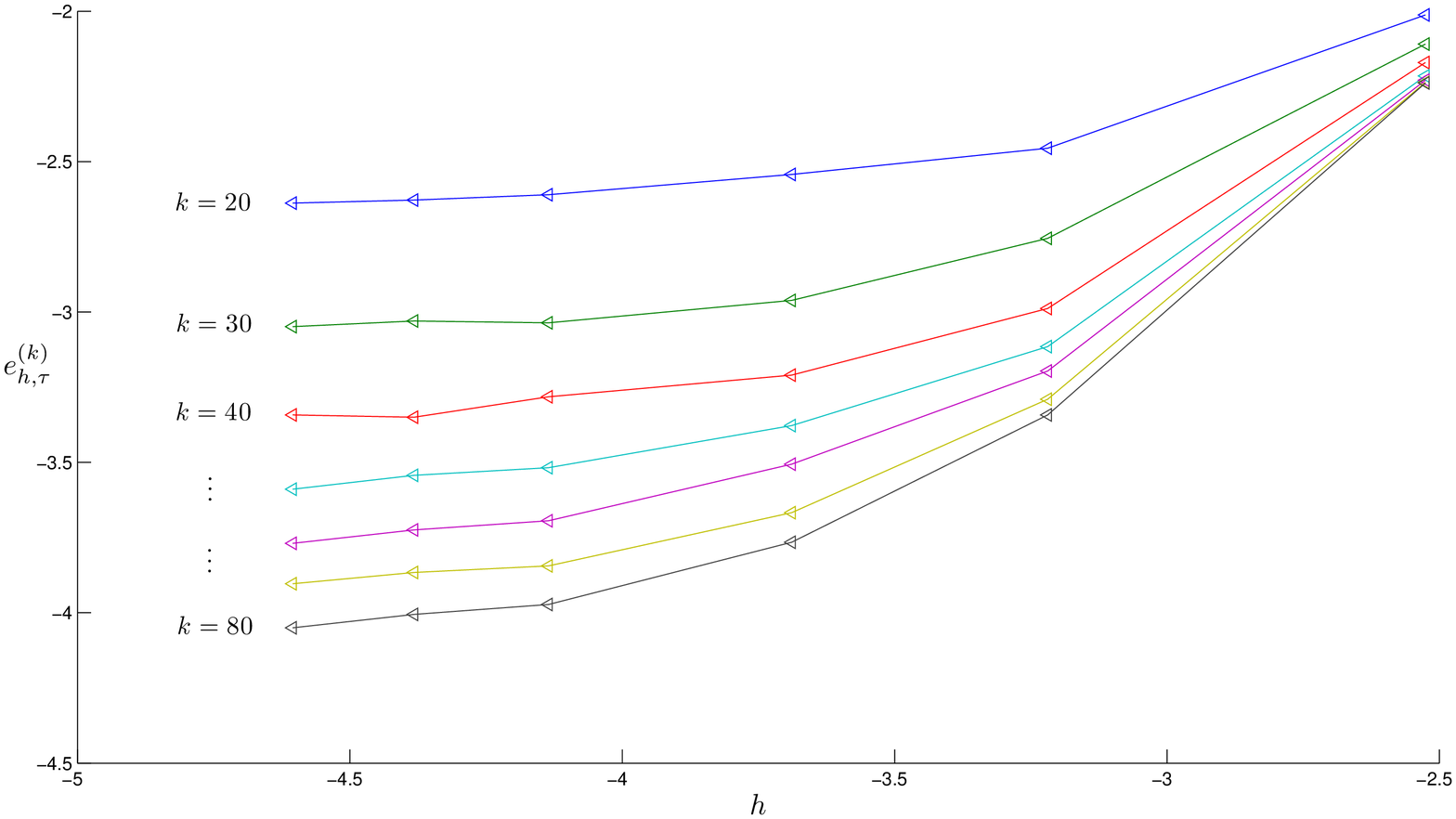}
\caption{Error between the computed profiles and the predicted profile in log-scale when $p = 5$.}
\label{fig:5p5}
\end{center}
\end{figure}
\begin{figure}[!htbp]
\begin{center}
\includegraphics[scale = 0.3]{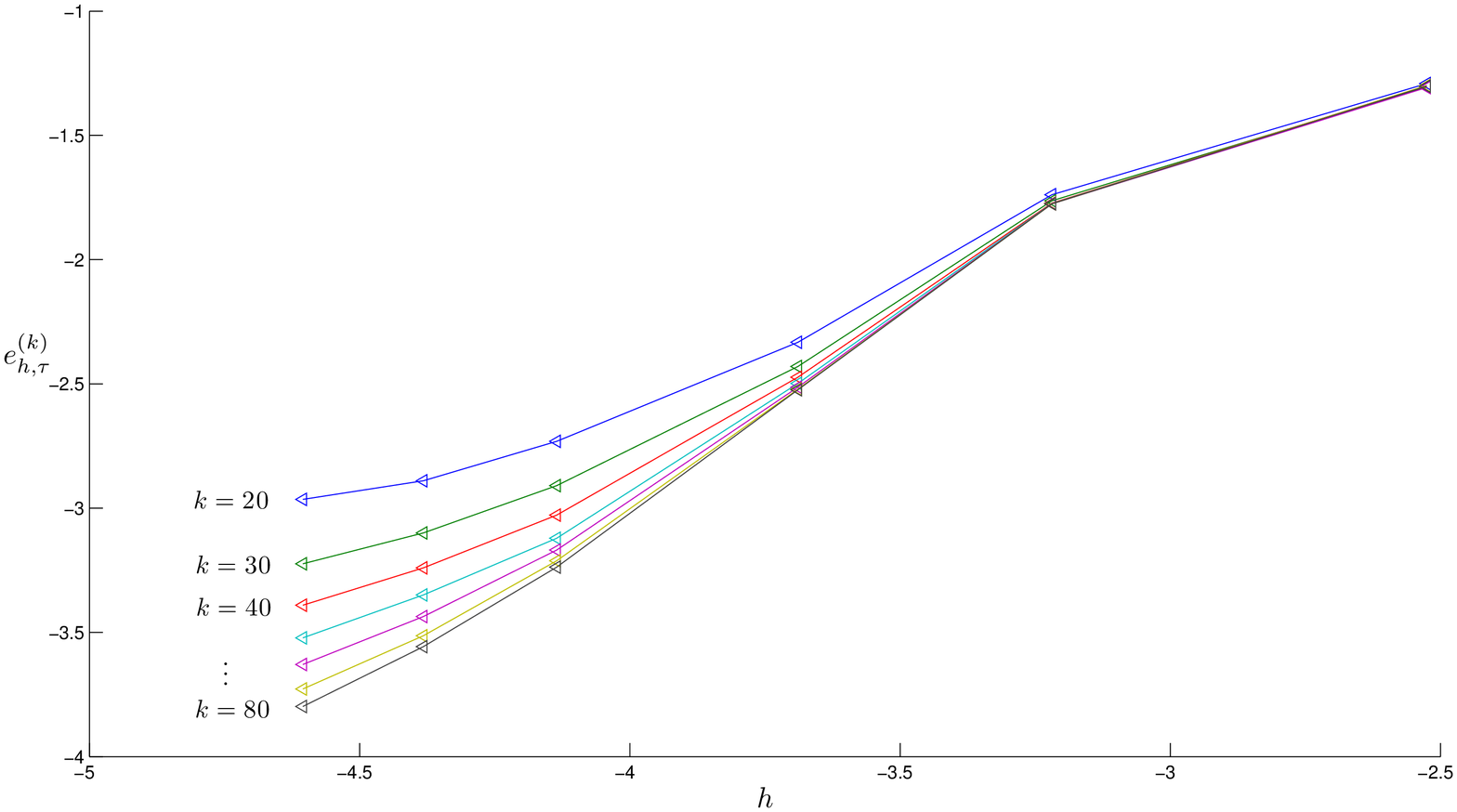}
\caption{Error between the computed profiles and the predicted profile in log-scale when $p = 7$.}
\label{fig:5p7}
\end{center}
\end{figure}

\subsection{The nonlinear heat equation in case $\beta \ne 0$}\label{sec:5.2}
\subsubsection{A formal calculation}
This part gives a formal calculation to obtain the prediction given in \eqref{equ:preproFgra}. This kind of arguments can be found in \cite{BKcpam88}, \cite{MZdm97} and \cite{MZjfa08}. Using similarity variables defined in \eqref{equ:simiWU} with $a = 0$, we see that $w = w_{0,T}$ satisfies the following equation for all $s\ge -\log{T}$ and $y \in \mathbb{R}^N$:
\begin{equation}\label{equ:selfw}
w_s = \Delta w - \frac{1}{2}y. \nabla w - \frac{w}{p-1} + |w|^{p-1}w + \beta |\nabla w|^{\frac{2p}{p+1}}.
\end{equation}
We try to find a solution of \eqref{equ:selfw} in the form $v\left(\frac{y}{\sqrt{s}}\right)$, with 
\begin{equation*}
v(0) = \kappa, \quad \lim_{|z| \to +\infty}|v(z)| = 0.
\end{equation*}
A computation shows that $v$ must satisfy the following equation, for each $s \geq -\log T$ and each $z \in \mathbb{R}^N$:
\begin{equation}\label{equ:Vz}
-\frac{z. \nabla v(z)}{2s} = \frac{1}{s} \Delta v(z) - \frac{z.\nabla v(z)}{2} - \frac{v(z)}{p-1} + |v(z)|^{p-1}v(z) +  \frac{\beta}{s^{\frac{p}{p+1}}}|\nabla v(z)|^{\frac{2p}{p+1}}.
\end{equation}
We formally seek regular solutions of \eqref{equ:selfw} in the form 
\begin{equation*}
V(z) = v_0\left(z\right) + \dfrac{1}{s^\alpha} R(z,s),
\end{equation*}
where $z = \frac{y}{\sqrt{s}}$, $\alpha > 0$ and $\|R\|_{L^\infty} \leq C$. \\
Pugging this ansatz in \eqref{equ:Vz} and making $s \to +\infty$, we obtain the following equation satisfied by $v_0$,
\begin{equation}\label{equ:v0}
-\frac{1}{2}zv_0'(z) - \frac{1}{p-1}v_0(z) + v_0(z)^p = 0.
\end{equation}
Solving \eqref{equ:v0} yields
\begin{equation}
v_0(z) = \left(p-1 + b z^2\right)^{-\frac{1}{p-1}},
\end{equation}
for some constant $b = b(\beta) \in \mathbb{R}$. We impose $b(\beta) > 0$ in order to have a bounded constant solution.
\begin{rema}
In the case $\beta = 0$, imposing an analyticity condition, Berger and Kohn \cite{BKcpam88} have formally found $b(0) = \frac{(p-1)^2}{4p}$, which is the coefficient of $f$ given in \eqref{equ:defFshe}. The value of $c_0$ was confirmed in several contributions (Filippas and Kohn \cite{FKcpam92}, Herrero and Vel\'azquez \cite{HVaihn93}, Bricmont and Kupiainen \cite{BKnon94}). Unfortunately, we were not able to adapt the formal approach of \cite{BKcpam88} in the case $\beta \ne 0$, so we only have a numerical expression of $\beta$ in Figure \ref{fig:BofBeta} below.
\end{rema}
\subsubsection{Numerical simulations}
An important aim in this work is to give a numerical confirmation for the conjectured profile given in \eqref{equ:preproGran}. Note that we have just given a formal argument in the previous subsection, for the existence of that profile, without, specifying the value of $b(\beta)$. Up to our knowledge, there is neither a rigorous proof nor a numerical confirmation for \eqref{equ:preproGran}, and our paper is the first to exhibit such a solution numerically. More importantly, thanks to our computations, we are able to find a numerical approximation of $b(\beta)$ in the formula of $\bar{f}_\beta$ in \eqref{equ:defFgra} from our computations. \\

If we make the same analysis to check that the numerical profile fits with the conjecture theoretical profile \eqref{equ:preproGran} as the above analysis when $\beta = 0$, then the same result holds in this case, namely
\begin{equation}\label{equ:numProGRA}
u^{(k)}(z \lambda^{-1}\xi_{k-1}^+, \tau_k^*) \sim M \left(1 + \left(\alpha^{1-p}- 1\right)\lambda^{-2} z^2 \right)^{-\frac{1}{p-1}},  \quad -1< z < 1.
\end{equation}
Figures \ref{fig:6betap5} and \ref{fig:6betap7} show the graphs of the computed profile $\mathbf{U_{h,\tau}^{(80)}}(z\lambda^{-1}\xi_{k-1}^+, \tau_k^*)$ and the predicted profile given in the right hand side of \eqref{equ:numProGRA}, for computations using $I = 320$, $\beta = 1$, $p = 5$ and $p = 7$.\\
\begin{figure}[!htbp]
\begin{center}
\includegraphics[scale = 0.3]{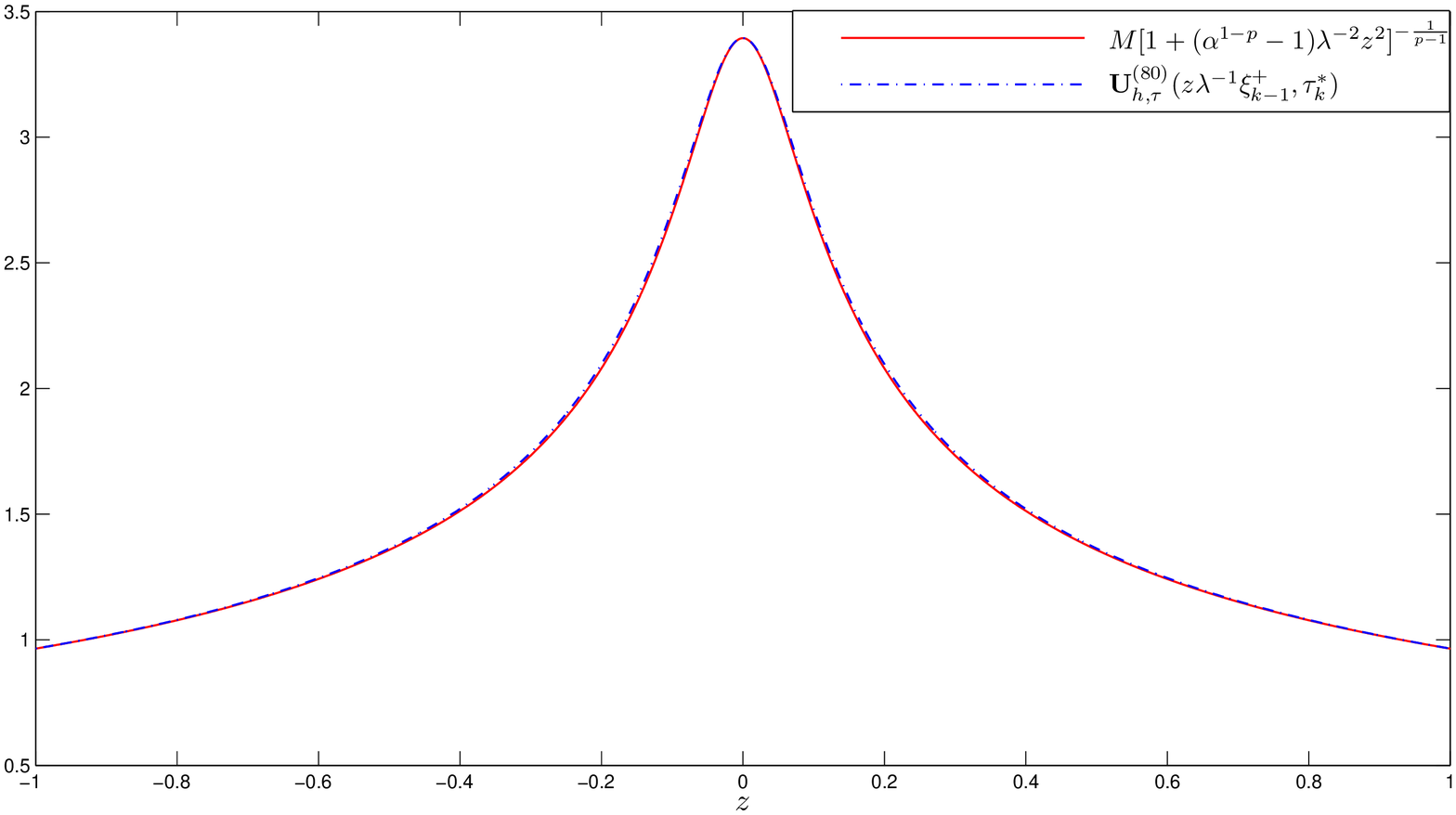}
\caption{The computed and the predicted profiles in \eqref{equ:numProGRA}, for computations using $I = 320$, $\beta = 1$ and $p = 5$.}
\label{fig:6betap5}
\end{center}
\end{figure}
\begin{figure}[!htbp]
\begin{center}
\includegraphics[scale = 0.3]{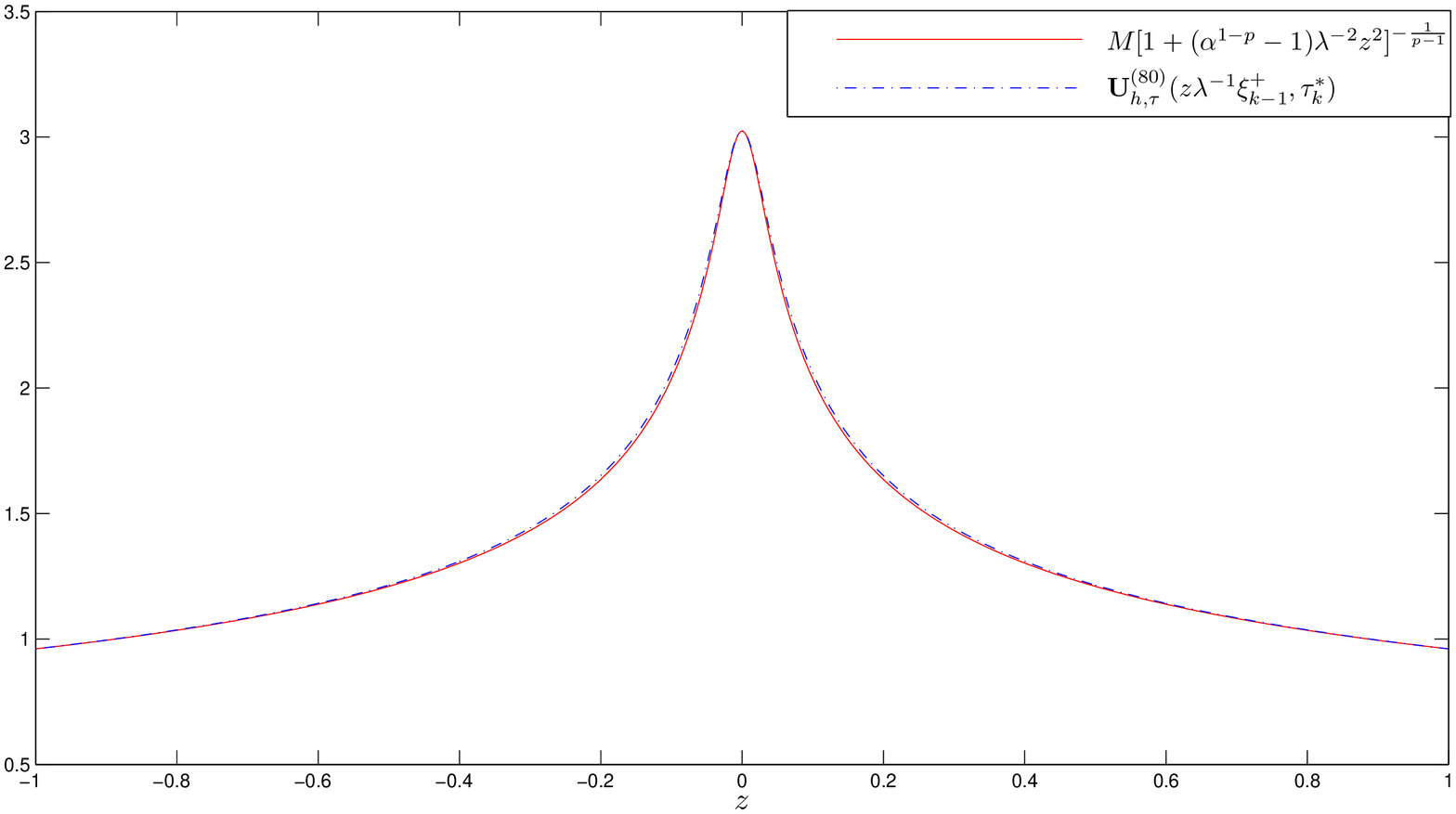}
\caption{The computed and the predicted profiles in \eqref{equ:numProGRA},  for computations using $I = 320$, $\beta = 1$ and $p = 7$.}
\label{fig:6betap7}
\end{center}
\end{figure}

In order to compute the value of $b(\beta)$ from the simulations, we use the relation \eqref{equ:relaUk_Wk} with $\xi_{k} = z\lambda^{-1}\xi_{k-1}^+$, we get
\begin{equation}\label{equ:grela12}
u^{(k)}(z\lambda^{-1}\xi_{k-1}^+, \tau_k^*) = \lambda^{\frac{2k}{p-1}} (T -t_k)^{-\frac{1}{p-1}}w \left(\lambda^k \frac{z\lambda^{-1}\xi_{k-1}^+}{\sqrt{T-t_k}}, s_k \right).
\end{equation}
We recall from \eqref{equ:defFgra} that the predicted profile $\bar{f}_\beta$ is given by
\begin{equation}\label{equ:theoProGRA}
\bar{f}_\beta(z) = \kappa \left(1 + \frac{b(\beta)}{p-1} z^2 \right)^{-\frac{1}{p-1}}, \quad z = \frac{x}{\sqrt{(T-t)|\log(T-t)|}}, \quad \kappa = (p-1)^{-\frac{1}{p-1}},
\end{equation}
and that
\begin{equation}\label{equ:inf3}
\sup_{|z| < K} \left| w(y,s) - \bar{f}_\beta(z) \right| \to 0 \quad \text{as } s \to \infty \quad \text{with } z = \frac{y}{\sqrt{s}}.
\end{equation}
From \eqref{equ:inf3}, \eqref{equ:theoProGRA} and \eqref{equ:grela12}, ignoring the error of asymptotic behavior as $s $ goes to infinity, we obtain
\begin{align*}
u^{(k)}(z\lambda^{-1}\xi_{k-1}^+, \tau_k^*) &= \lambda^{\frac{2k}{p-1}} (T -t_k)^{-\frac{1}{p-1}}f \left(\lambda^k \frac{z\lambda^{-1}\xi_{k-1}^+}{\sqrt{T-t_k}}\times \frac{1}{\sqrt{s_k}} \right)\\
&= \lambda^{\frac{2k}{p-1}} (T -t_k)^{\frac{-1}{p-1}}\kappa\left(1 +  \frac{b(\beta)}{p-1} \frac{\lambda^{2k} z^2 \lambda^{-2} (\xi_{k-1}^+)^2}{T - t_k} \frac{1}{s_k}\right)^{-\frac{1}{p-1}}.
\end{align*}
After some straightforward calculations, we arrive at
\begin{equation*}
b(\beta) = \frac{s_k}{(\xi_{k-1}^+)^2}\left[ \frac{\kappa \lambda^{2k}(p-1)[u^{(k)}(z\lambda^{-1}\xi_{k-1}^+, \tau_k^*)]^{1-p} - (p-1)(T-t_k)}{\lambda^{2k-2} z^2} \right].
\end{equation*}
Setting $z = \lambda$ and taking the limit of the above equation as $k$ goes to infinity, we get
\begin{equation*}
b(\beta) = \lim_{k \to +\infty}\frac{s_k}{(\xi_{k-1}^+)^2}\zeta_k, 
\end{equation*}
where $$\zeta_k = (p-1)\left( \kappa [u^{(k)}(\xi_{k-1}^+, \tau_k^*)]^{1-p} - \lambda^{-2k}(T-t_k)\right).$$
Using \eqref{equ:T_tk} and \eqref{equ:numProGRA}, we see that $\zeta_k$ approaches a limit given by
$$\lim_{k \to + \infty}\zeta_k = M^{1-p}\left[(p-1)\kappa\alpha^{1-p} -1\right].$$
This implies that the ratio $\frac{s_k}{(\xi_{k-1}^+)^2}$ should approach a constant as $k$ tends to infinity. This is presented in Figure \ref{fig:7appRat}.
\begin{figure}[!htbp]
\begin{center}
\includegraphics[scale = 0.3]{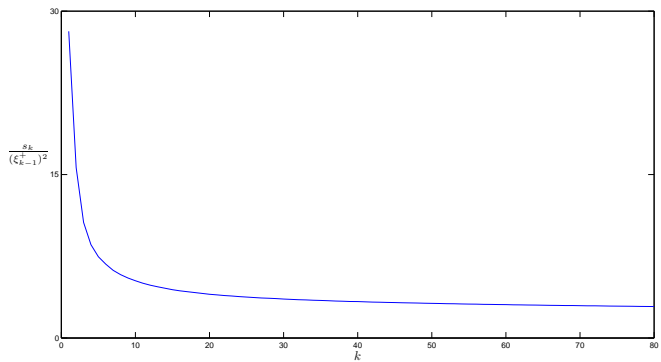}
\caption{The graph of $\frac{s_k}{(\xi_{k-1}^+)^2}$ versus $k$, for computations using $I = 320$, $p = 5$ and $\beta = 0$.}
\label{fig:7appRat}
\end{center}
\end{figure}
We remark that the computations of $s_k$ and $\zeta_k$ do not depend on $\beta$. 
Moreover, we know that the value of $b(0)$ is $\frac{(p-1)^2}{4p}$. In particular, we compute the value of $b(\beta)$ by
\begin{equation*}
b(\beta) = \frac{C_K}{\left[\xi_{K-1}^+(\beta)\right]^2},
\end{equation*}
where $C_K = b(0)\left[\xi_{K-1}^+(0)\right]^2$ for $K$ large.
\begin{figure}[!htbp]
\begin{center}
\includegraphics[scale = 0.3]{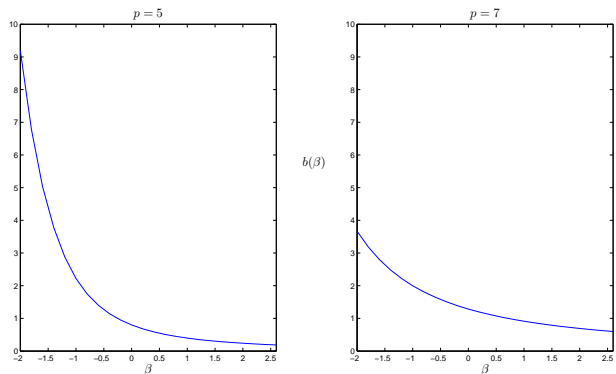}
\caption{The computed values of $b(\beta)$. (Left) $p =5$. (Right) $p = 7$.}
\label{fig:BofBeta}
\end{center}
\end{figure}

\noindent Consequently, we have just given a numerical evidence for the following conjecture:
\begin{conj} Equation \eqref{equ:semiHeatEqu} has a solution $u(x,t)$ which blows up in finite time $T$ with 
\begin{equation}\label{equ:conjFb}
\sup_{|z|<K}|(T-t)^{1/(p-1)}u(x,t) - \bar{f_\beta}(z)| \to 0, \quad \text{as}\quad  t \to T,
\end{equation}
where $K > 0, z= \frac{x}{\sqrt{(T-t)|\log(T-t)|}}$ and 
\begin{equation*}
\bar{f_\beta}(z) = \left(p-1+b(\beta)|z|^2\right)^{-\frac{1}{p-1}}, \quad \text{with}\quad b(0) = \frac{(p-1)^2}{4p},
\end{equation*}
and $b(\beta)$ is represented in Figure \ref{fig:BofBeta}, for $p = 5$ and $p = 7$.
\end{conj}

\noindent While remarking numerical simulation for equation \eqref{equ:semiHeatEqu} with $\beta \ne 0$, we could never obtain the self-similar behavior \eqref{equ:sesi} rigorously proved in \cite{STWiumj96}. On the contrary, we could exhibit the behavior \eqref{equ:conjFb}, at the heart of our conjecture. In our opinion, this is probably due to the fact that the behavior \eqref{equ:sesi} is unstable, unlike the behavior \eqref{equ:conjFb}, which we suspect to be stable with respect to perturbations in initial data. 

\subsection{The complex Ginzburg-Landau equation}
We recall that $e^{\imath \theta} \tilde{f}_{\delta, \gamma}$ is an asymptotic profile of the solution of \eqref{equ:GL} where $\theta \in \mathbb{R}$ and $\tilde{f}_{\delta, \gamma}$ is given in \eqref{equ:funcfpGL}, namely 
\begin{equation}\label{equ:proGLf2}
\tilde{f}_{\delta, \gamma} = \left(p-1 + b(\delta, \gamma) |z|^2\right)^{-\frac{1 + \imath \delta}{p-1}}, \quad b(\delta, \gamma) = \frac{(p-1)^2}{4(p-\delta^2 - \gamma\delta - \gamma\delta p)} > 0.
\end{equation}
Using the same analysis as Section \ref{sec:anaSHE} resulting \eqref{equ:profnumSHE}, we have for $|z| < 1$,
\begin{equation}\label{equ:profnumGL}
u^{(k)}(z\lambda^{-1}y_{k-1}^+, \tau_k^*) \sim M^{1 + \imath \delta} \lambda^{-\frac{2\imath k \delta}{p-1}} (p-1) ^\frac{\imath \delta}{ p-1} e^{\imath \theta} \left(1 + (\alpha^{1-p} - 1)\lambda^{-2}z^2\right)^{-\frac{1 + \imath \delta}{p-1}}.
\end{equation}
\begin{rema}
We remark that the rescaled profile \eqref{equ:profnumGL} is obtained under the assumption $p - \delta^2 - \gamma\delta (p+1) > 0$. If this condition is not satisfied, the question is open. 
\end{rema}
\begin{rema} If we take the modulus and the phase of both sides in \eqref{equ:profnumGL}, then we get
\begin{equation}\label{equ:modulusGL}
\left|u^{(k)}\right|(z\lambda^{-1}y_{k-1}^+, \tau_k^*) \sim M (1 + (\alpha^{1-p} - 1)\lambda^{-2}z^2)^{-\frac{1}{p-1}}, \quad |z| < 1,
\end{equation}
\begin{align}
phase\left[u^{(k)}\right](z\lambda^{-1}y_{k-1}^+, \tau_k^*) &\sim \theta + \frac{\delta}{p-1}\left(\ln M + \ln (p-1) - 2k \ln \alpha \right) \nonumber\\
& - \frac{\delta}{p-1}\ln \left(1 + (\alpha^{1-p} - 1)\lambda^{-2}z^2 \right), \quad |z| < 1. \label{equ:phaseGL}
\end{align}
The right hand side of \eqref{equ:modulusGL} is the same as in \eqref{equ:profnumSHE}.
\end{rema}

\subsubsection{Experiments with  $p - \delta^2 - \gamma\delta (p+1) > 0$.}
We first make an experiment with $\gamma= 0, \,\delta = 0.2, \, p = 5$ and the initial grid with $I = 320$. The numerical result displayed in Figure \ref{fig:GLb0d02} is in agreement with the expectation obtained in \eqref{equ:modulusGL} and \eqref{equ:phaseGL}. Both the numerical modulus and phase coincide with the predicted profile given in \eqref{equ:modulusGL} and \eqref{equ:phaseGL} within plotting resolution. 
\begin{figure}[!htbp]
\begin{center}
\includegraphics[scale = 0.3]{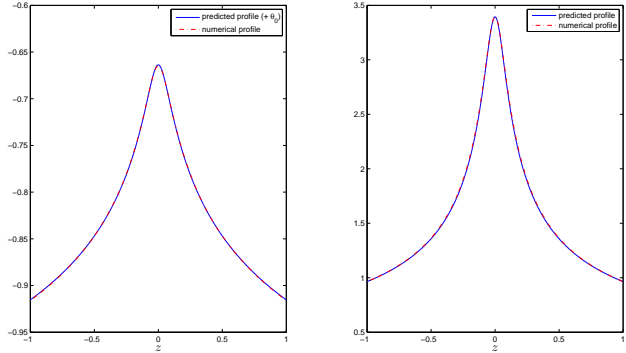}
\caption{Comparing the numerical profile with the predicted profile given in \eqref{equ:modulusGL} and \eqref{equ:phaseGL} after 80 iterative steps ($\gamma = 0, \delta =0.2, p =5, I = 320$). (Left) $phase \left[u^{(k)}\right](z \lambda^{-1} y^+_{k-1}, \tau_k^*)$. (Right) $\left|u^{(k)}\right|(z \lambda^{-1} y^+_{k-1}, \tau_k^*)$.}
\label{fig:GLb0d02}
\end{center}
\end{figure}

An experiment with $\gamma= 0, \, p =5$ and various values of $\delta$ are performed on three grids with $I = 100, 200, 320$. The purpose is to confirm the theoretical profile $\tilde{f}_{\delta, \gamma}$ given in \eqref{equ:proGLf2}. More precisely, we would like to calculate values of $b(\delta, 0)$ from our numerical simulation. We recall that the theoretical value of $b(\delta,0)$ is equal to $\frac{(p-1)^2}{p - \delta^2}$. In Figure \ref{fig:GLvalueB}, we have the computed values of $b(\delta,0)$ on various initial grids $I$. Note that these computed values tend to the predicted ones as $I$ increases. However, as $\delta$ approaches $\sqrt{p}$ ($\sqrt{5}$ in Figure \ref{fig:GLvalueB}), $b$ becomes singular, and that is the reason why the coincidence between the numerical and theoretical values becomes less clear.
\begin{figure}[!htbp]
\begin{center}
\includegraphics[scale = 0.3]{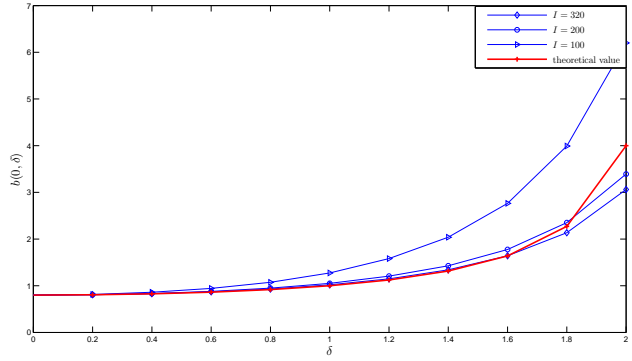}
\caption{The computed values of $b(\delta,0)$ for various initial grids when $ p = 5$.}
\label{fig:GLvalueB}
\end{center}
\end{figure}

A further experiment with $\gamma = 1, \, \delta = 1$ is shown in Figure \ref{fig:GLb1d1}. These calculations show the relationship we obtained in \eqref{equ:modulusGL} and \eqref{equ:phaseGL}. Both the numerical phase and modulus coincide with the predicted ones given in \eqref{equ:modulusGL} and \eqref{equ:phaseGL} within plotting resolution.

\begin{figure}[!htbp]
\begin{center}
\includegraphics[scale = 0.3]{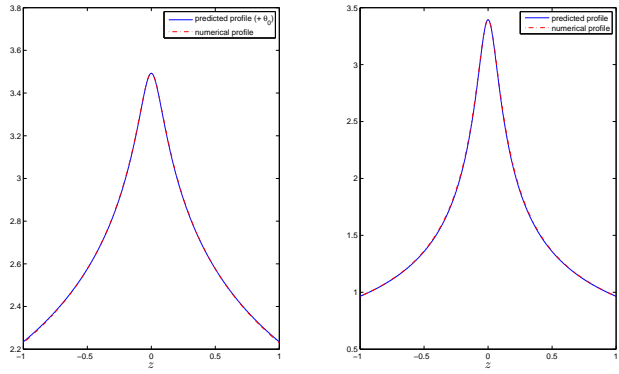}
\caption{Comparing the numerical profile with the predicted profile given in \eqref{equ:modulusGL} and \eqref{equ:phaseGL} after 80 iterative steps ($\gamma = 1, \delta =1, p =5, I = 320$). (Left) $phase \left[u^{(k)}\right](z \lambda^{-1} y^+_{k-1}, \tau_k^*)$. (Right) $\left|u^{(k)}\right|(z \lambda^{-1} y^+_{k-1}, \tau_k^*)$.}
\label{fig:GLb1d1}
\end{center}
\end{figure}

\subsubsection{Experiments with  $p - \delta^2 - \gamma\delta (p+1) < 0$.}
In this section, we make some experiments with $\gamma = 0$ and $\delta > \sqrt{p} = \sqrt{5}$. For $\delta$ large enough, there is no blow-up phenomenon (for example with $\delta = 3$). With $\delta$ near $\sqrt{p}$, we made two simulations with $\delta = \sqrt{p} + 0.1$ and $\delta = \sqrt{p} + 0.5$, then the blow-up phenomenon still occurs. Figure \ref{fig:GLb0d01} displays the modulus of $u^{(k)}(z\lambda^{-1}y_{k-1}^+, \tau_k^*)$ at some selected values of $k$, for computations using the initial grid $I = 320$. It shows the rescaled profile $z \mapsto \left|u^{(k)}\right|(z\lambda^{-1}y_{k-1}^+, \tau_k^*)$. We can see that these rescaled profiles converge as $k $ increases.\\

\begin{figure}[!htbp]
\begin{center}
\includegraphics[scale= 0.3]{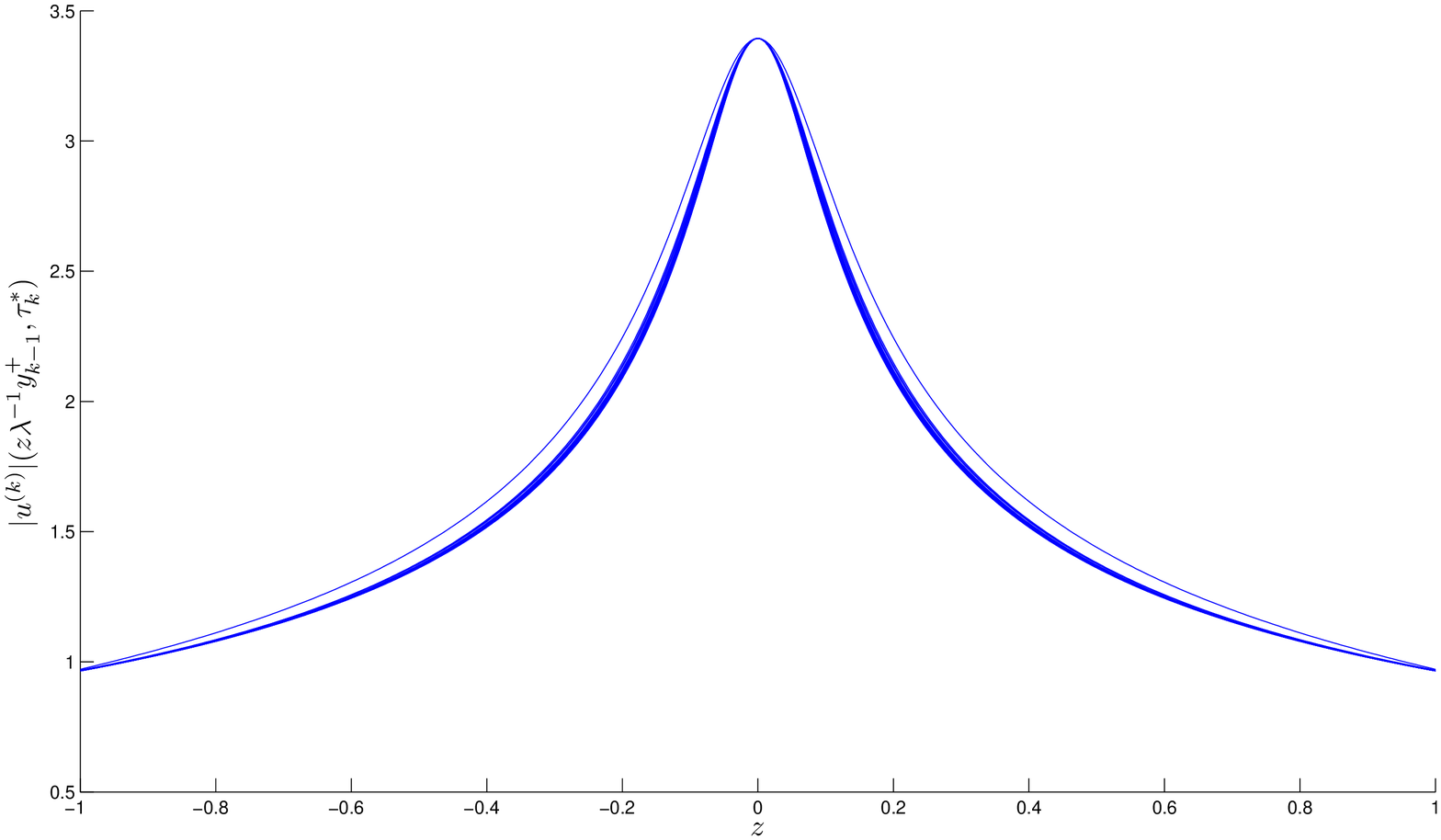}
\includegraphics[scale= 0.3]{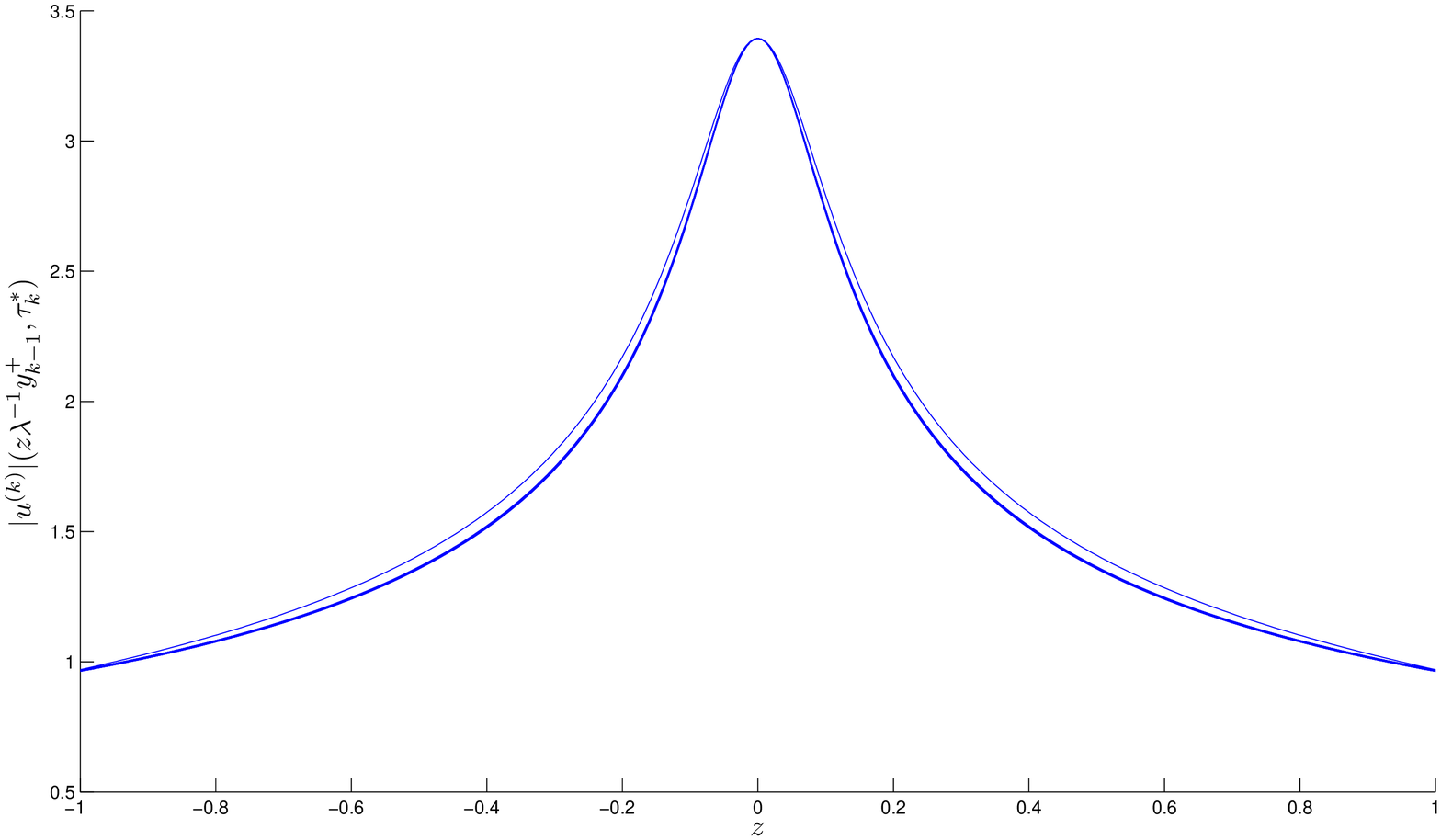}
\caption{The numerical values of $u^{(k)}(z\lambda^{-1}y_{k-1}^+, \tau_k^*)$ at some selected values of $k$, for computation using the initial grid $I = 320$ with $p=5$. (Above) $\gamma = 0, \delta = \sqrt{p} + 0.1$. (Below) $\gamma = 0, \delta =\sqrt{p} + 0.5$.}
\label{fig:GLb0d01}
\end{center}
\end{figure}
\noindent Consequently, if $p - \delta^2 - \gamma \delta - \gamma \delta p < 0$, the blow-up phenomenon may occur and there may exist a blow-up profile. So far, we have no answer for this case. We wonder whether its solution behaves as the solution in the case $p - \delta^2 - \gamma \delta - \gamma \delta p > 0$ with a different function of $b(\delta, \gamma)$ in formula of $\tilde{f}_{\delta, \gamma}$ given in \eqref{equ:proGLf2}.

\appendix
\renewcommand*{\thesection}{\Alph{section}}
\counterwithin{theo}{section}
\section{A regularity result for equation \eqref{equ:semiHeatEqu}} \label{ap:A}
\noindent We claim the following:
\begin{prop}[\textbf{Parabolic regularity}]\label{prop:reg}
Consider $u$ solution of 
\begin{equation} \label{eq:p}
\left\{\begin{array}{lll}
u_t &=  u_{xx} + |u|^{p - 1}u + \beta |u_x|^q,\quad & \text{in}\quad \Omega \times (0,T),\\
u(x,t)&= 0 \quad & \text{on}\quad \partial \Omega \times [0,T),\\
u(x,0) &= u_0(x), \quad & \text{on}\quad \bar{\Omega}.
\end{array}
\right.
\end{equation}
where $u(t): x \in \Omega \mapsto \mathbb{R}$ with $\Omega$ is an interval in $\mathbb{R}$, $p, q > 1$ and $\beta \in \mathbb{R}$.\\
Assume that $\|u_0\|_{\mathcal{C}^2(\bar{\Omega})} \leq C_0$ and $\|u\|_{\mathcal{C}(\bar{\Omega}\times [0, T_0])} + \|u_x\|_{\mathcal{C}(\bar{\Omega}\times [0, T_0])}\leq C_u$ with $T_0 < T$ ($T$ is the existence time of the maximal solution).  Then for all $t \in [0, T_0]$,\\
$i) \|u_{xx}(t)\|_{L^\infty(\bar{\Omega})} + \|u_t(t)\|_{L^\infty(\bar{\Omega})} \leq C$ for some $C = C(C_0, C_u, T_0, p, q, \beta)$.\\
Assume in addition, $\|u_0\|_{\mathcal{C}^4(\bar{\Omega})} \leq C_0$ and $p, q \geq 2$. Then for all $t \in [0, T_0]$,\\
$ii) \quad \|u_{xxx}(t)\|_{L^\infty(\bar{\Omega})} + \|u_{xxxx}(t)\|_{L^\infty(\bar{\Omega})} + \|u_{tt}(t)\|_{L^\infty(\bar{\Omega})} \leq C$ for some $C = C(C_0, C_u, T_0, p, q, \beta)$.
\end{prop}
\begin{rema} Chipot and Weissler showed in \cite{CWjma89} (see Proposition 2.2) that for $s \in \mathbb{R}$ sufficient large and $u_0 \in W_0^{1,s}(\Omega)$, then $u$, a solution of \eqref{eq:p}, satisfies that
$$\|u(t)\|_{L^\infty} \quad \text{and} \quad \|u_x(t)\|_{L^\infty} \quad \text{are bounded for any inteval $[0,T_0]$ with $T_0 < T$}.$$
\end{rema}
\begin{proof}[\textbf{Proof of Proposition \ref{prop:reg}}] In what follows, we write $\|\cdot\|_\infty = \|\cdot\|_{L^\infty(\bar{\Omega})}$ for simplicity and denote by $C_1, C_2,\dots$ constants depending only on $C_0, C_u, T_0, p, q$ and $\beta$.\\

\noindent $i)$ We see from \eqref{eq:p} that $\|u_t(t)\|_{\infty}$ is bounded on $[0,T_0]$ if $\|u_{xx}(t)\|_\infty$ is bounded on $[0, T_0]$. Let us consider $h = u_{xx}$, then $h$ satisfies
\begin{equation}\label{eq:hA}
h_t = h_{xx} + \partial_x\left(p|u|^{p-1}u_x + q\beta|u_x|^{q - 2}u_x h\right).
\end{equation}
An integral form of the solution of equation \eqref{eq:hA} is 
\begin{equation}\label{eq:hA1}
h(t) = e^{t\Delta }h(0) + \int_0^te^{(t -s)\Delta}\partial_x\left(p|u(s)|^{p-1}u_x(s) + q\beta|u_x(s)|^{q - 2}u_x(s) h(s)\right)ds,
\end{equation}
where $e^{t\Delta}$ denotes the heat semigroup on $\Omega$ with Dirichlet boundary condition.\\
Recall that for all $\varphi \in L^\infty$, 
\begin{equation}\label{eq:rePa}
\|e^{t\Delta}\varphi\|_\infty\leq \|\varphi\|_\infty \quad \text{and}\quad \|e^{t\Delta}\nabla \varphi\|_\infty \leq \frac{C'}{\sqrt{t}}\|\varphi\|_\infty.
\end{equation}
Since $u_0 \in \mathcal{C}^2$ and $\|u(t)\|_\infty, \|u_x(t)\|_\infty$ are bounded for all $t \in [0, T_0]$, then we have by \eqref{eq:rePa} and \eqref{eq:hA1} that
$$\|h(t)\|_\infty \leq C_1 + C_1\int_0^t\frac{\|h(s)\|_\infty}{\sqrt{t-s}} ds, \quad \forall t \in [0, T_0].$$
Using a Growall's argument, we have
$$\|h(t)\|_\infty \leq 2C_1e^{C_1\sqrt{T_0}}, \quad \forall t \in [0, T_0].$$
Therefore, $\|u_{xx}(t)\|_\infty$ is bounded for all $t \in [0, T_0]$ which concludes the proof of $i)$.\\

\noindent $ii)$ We assume additionally in what follows that $p, q \geq 2$, $\|u_0\|_{\mathcal{C}^4(\bar{\Omega})} \leq C_0$. Consider $v = u_{xxx}$, let us show that $\|v(t)\|_\infty$ is bounded for all $t \in [0, T_0]$.  From \eqref{eq:p} we see that $v$ satisfies the following equation
\begin{equation}\label{eq:vAp}
v_t = v_{xx} + p|u|^{p-1}v + \beta q\partial_x(|u_x|^{q-2}u_x v) + \phi + \partial_x \psi,
\end{equation}
where 
\begin{align*}
\phi &= p(p-1)|u|^{p-3}uu_xu_{xx},
\psi &= p(p-1)|u|^{p-3}u(u_x)^2 + \beta q(q-1)|u_x|^{q-2}(u_{xx})^2.
\end{align*}
We now use an integral formulation of \eqref{eq:vAp} to write
\begin{align*}
v(t) = e^{t\Delta}v(0) &+ p\int_0^te^{(t -s)\Delta}|u(s)|^{p-1}v(s)ds \\
&+ \beta q \int_0^te^{(t-s)\Delta}\partial_x(|u_x(s)|^{q-2}u_x(s) v(s))ds\\
&+ \int_0^te^{(t-s)\Delta}\phi(s)ds + \int_0^te^{(t-s)\Delta}\partial_x\psi(s)ds.
\end{align*}
From $(i)$ and the hypothesis on $u_0 \in \mathcal{C}^4$, we see that for all $t \in [0, T_0]$,
$$\|v(0)\|_\infty + \|u(t)\|_\infty^{p-1} + \|u_x(t)\|_\infty^{q-1} + \|\phi(t)\|_\infty + \|\psi(t)\|_\infty \leq C_2.$$
Hence, the use of \eqref{eq:rePa} yields
\begin{align*}
\|v(t)\|_\infty \leq C_2 + C_2\int_0^t \left(1 + \frac{1}{\sqrt{t-s}}\right)\|v(s)\|_\infty ds \leq 2C_2e^{C_2(T_0 + 2\sqrt{T_0})}, \quad \forall t \in [0, T_0],
\end{align*}
which follows that $\|u_{xxx}(t)\|_\infty$ is bounded on $[0, T_0]$.\\

\noindent We now bound $\|u_{tt}(t)\|_\infty$ on $[0, T_0]$. Consider $\theta = u_{tt}$, by \eqref{eq:p}, we see that $\theta$ satisfies
\begin{equation}\label{eq:theA}
\theta_t = \theta_{xx} + \eta \theta +  \beta q \partial_x\left(|u_x|^{q-2}u_{x} \theta \right) + \gamma,
\end{equation}
where
\begin{align*}
\eta & = p|u|^{p-1} - \beta q(q - 1)|u_x|^{q-2}u_{xx},\\
\gamma & = p(p-1)|u|^{p-3}u(u_t)^2 +\beta q(q-1)|u_x|^{q-2}\left(u_{xxx} + p|u|^{p-1}u_x + \beta q |u_x|^{q-2}u_x u_{xx}\right)^2.
\end{align*}
An integral form of the solution of equation \eqref{eq:theA} is
\begin{align*}
\theta(t) = e^{t\Delta }\theta(0) &+ \int_0^te^{(t - s)\Delta}\eta(s)\theta(s)ds \\
&+ \beta q\int_0^t e^{(t - s)\Delta}\partial_x\left(|u_x(s)|^{q-2}u_{x}(s)\theta(s) \right)ds + \int_{0}^te^{(t-s)\Delta}\gamma(s)ds.
\end{align*}
Since $u_0 \in \mathcal{C}^4$, then $\|\theta(0)\|_\infty = \|u_{tt}(0)\|_\infty$ is bounded. Using the fact that $\|u_{xxx}(t)\|_\infty$ is bounded on $[0,T_0]$ and $(i)$, we have by \eqref{eq:rePa} that  
$$\|\theta(t)\|_\infty \leq C_3 + C_3\int_0^t\left(1 + \frac{1}{\sqrt{t-s}}\right)\|\theta(s)\|_\infty ds \leq 2C_3e^{C_3(T_0 + 2\sqrt{T_0})}, \quad \forall t \in [0, T_0].$$
\noindent Since $\|u_{tt}(t)\|_{L^\infty}$ is bounded on $[0, T_0]$, we have from \eqref{eq:p} that $\|u_{xxxx}(t)\|_{L^\infty}$ is also bounded on $[0,T_0]$. This completes the proof of Proposition \ref{prop:reg}.
\end{proof}

\def\cprime{$'$}


\begin{thebibliography}{53}
\expandafter\ifx\csname natexlab\endcsname\relax\def\natexlab#1{#1}\fi
\expandafter\ifx\csname url\endcsname\relax
  \def\url#1{\texttt{#1}}\fi
\expandafter\ifx\csname urlprefix\endcsname\relax\def\urlprefix{URL }\fi

\bibitem[{Abia et~al.(1996)Abia, L{\'o}pez-Marcos, and Mart{\'i}nez}]{ALManm96}
Abia, L.~M., L{\'o}pez-Marcos, J.~C., Mart{\'i}nez, J., 1996. Blow-up for
  semidiscretizations of reaction-diffusion equations. Appl. Numer. Math.
  20~(1-2), 145--156, workshop on the method of lines for time-dependent
  problems (Lexington, KY, 1995).
\newline\urlprefix\url{http://dx.doi.org/10.1016/0168-9274(95)00122-0}

\bibitem[{Abia et~al.(1998)Abia, L{\'o}pez-Marcos, and Mart{\'i}nez}]{ALManm98}
Abia, L.~M., L{\'o}pez-Marcos, J.~C., Mart{\'i}nez, J., 1998. On the blow-up
  time convergence of semidiscretizations of reaction-diffusion equations.
  Appl. Numer. Math. 26~(4), 399--414.
\newline\urlprefix\url{http://dx.doi.org/10.1016/S0168-9274(97)00105-0}

\bibitem[{Abia et~al.(2001)Abia, L{\'o}pez-Marcos, and Mart{\'i}nez}]{ALManm01}
Abia, L.~M., L{\'o}pez-Marcos, J.~C., Mart{\'i}nez, J., 2001. The {E}uler
  method in the numerical integration of reaction-diffusion problems with
  blow-up. Appl. Numer. Math. 38~(3), 287--313.
\newline\urlprefix\url{http://dx.doi.org/10.1016/S0168-9274(01)00035-6}

\bibitem[{Acosta et~al.(2002{\natexlab{a}})Acosta, Dur{\'a}n, and
  Rossi}]{ADRcomp02}
Acosta, G., Dur{\'a}n, R.~G., Rossi, J.~D., 2002{\natexlab{a}}. An adaptive
  time step procedure for a parabolic problem with blow-up. Computing 68~(4),
  343--373.
\newline\urlprefix\url{http://dx.doi.org/10.1007/s00607-002-1449-x}

\bibitem[{Acosta et~al.(2002{\natexlab{b}})Acosta, {Fern{\'a}ndez Bonder},
  Groisman, and Rossi}]{AFGRdcds02}
Acosta, G., {Fern{\'a}ndez Bonder}, J., Groisman, P., Rossi, J.~D.,
  2002{\natexlab{b}}. Numerical approximation of a parabolic problem with a
  nonlinear boundary condition in several space dimensions. Discrete Contin.
  Dyn. Syst. Ser. B 2~(2), 279--294.
\newline\urlprefix\url{http://dx.doi.org/10.3934/dcdsb.2002.2.279}

\bibitem[{Assal{\'e} et~al.(2008)Assal{\'e}, K., and Diabate}]{ABNjam08}
Assal{\'e}, L.~A., K., B.~T., Diabate, N., 2008. Numerical blow-up time for a
  semilinear parabolic equation with nonlinear boundary conditions. Journal of
  Applied Mathematics 2008, Article ID 753518, 29 p.--Article ID 753518, 29 p.
\newline\urlprefix\url{http://eudml.org/doc/45748}

\bibitem[{Ball(1977)}]{BALjmo77}
Ball, J.~M., 1977. Remarks on blow-up and nonexistence theorems for nonlinear
  evolution equations. Quart. J. Math. Oxford Ser. (2) 28~(112), 473--486.

\bibitem[{Baruch et~al.(2010)Baruch, Fibich, and Gavish}]{BFGpd10}
Baruch, G., Fibich, G., Gavish, N., 2010. Singular standing-ring solutions of
  nonlinear partial differential equations. Phys. D 239~(20-22), 1968--1983.
\newline\urlprefix\url{http://dx.doi.org/10.1016/j.physd.2010.07.009}

\bibitem[{Berger and Kohn(1988)}]{BKcpam88}
Berger, M., Kohn, R.~V., 1988. A rescaling algorithm for the numerical
  calculation of blowing-up solutions. Comm. Pure Appl. Math. 41~(6), 841--863.
\newline\urlprefix\url{http://dx.doi.org/10.1002/cpa.3160410606}

\bibitem[{Bricmont and Kupiainen(1994)}]{BKnon94}
Bricmont, J., Kupiainen, A., 1994. Universality in blow-up for nonlinear heat
  equations. Nonlinearity 7~(2), 539--575.
\newline\urlprefix\url{http://stacks.iop.org/0951-7715/7/539}

\bibitem[{Cazenave et~al.(2013)Cazenave, Dickstein, and Weissler}]{CDWjma13}
Cazenave, T., Dickstein, F., Weissler, F., 2013. Finite-time blowup for a
  complex ginzburg--landau equation. SIAM Journal on Mathematical Analysis
  45~(1), 244--266.
\newline\urlprefix\url{http://epubs.siam.org/doi/abs/10.1137/120878690}

\bibitem[{Chen(1986)}]{CHEjut86}
Chen, Y.~G., 1986. Asymptotic behaviours of blowing-up solutions for finite
  difference analogue of {$u_t=u_{xx}+u^{1+\alpha}$}. J. Fac. Sci. Univ. Tokyo
  Sect. IA Math. 33~(3), 541--574.

\bibitem[{Chen(1992)}]{CHEhmj92}
Chen, Y.~G., 1992. Blow-up solutions to a finite difference analogue of
  {$u_t=\Delta u+u^{1+\alpha}$} in {$N$}-dimensional balls. Hokkaido Math. J.
  21~(3), 447--474.

\bibitem[{Chipot and Weissler(1989)}]{CWjma89}
Chipot, M., Weissler, F.~B., 1989. Some blowup results for a nonlinear
  parabolic equation with a gradient term. SIAM J. Math. Anal. 20~(4),
  886--907.
\newline\urlprefix\url{http://dx.doi.org/10.1137/0520060}

\bibitem[{Duran et~al.(1998)Duran, Etcheverry, and Rossi}]{DERdcds98}
Duran, R.~G., Etcheverry, J.~I., Rossi, J.~D., 1998. Numerical approximation of
  a parabolic problem with a nonlinear boundary condition. Discrete Contin.
  Dynam. Systems 4~(3), 497--506.
\newline\urlprefix\url{http://dx.doi.org/10.3934/dcds.1998.4.497}

\bibitem[{Ebde and Zaag(2011)}]{EZsema11}
Ebde, M.~A., Zaag, H., 2011. Construction and stability of a blow up solution
  for a nonlinear heat equation with a gradient term. S$\vec{\rm e}$MA J.~(55),
  5--21.

\bibitem[{{Fermanian Kammerer} et~al.(2000){Fermanian Kammerer}, Merle, and
  Zaag}]{FMZma00}
{Fermanian Kammerer}, C., Merle, F., Zaag, H., 2000. Stability of the blow-up
  profile of non-linear heat equations from the dynamical system point of view.
  Math. Ann. 317~(2), 347--387.
\newline\urlprefix\url{http://dx.doi.org/10.1007/s002080000096}

\bibitem[{{Fermanian Kammerer} and Zaag(2000)}]{FZnon00}
{Fermanian Kammerer}, C., Zaag, H., 2000. Boundedness up to blow-up of the
  difference between two solutions to a semilinear heat equation. Nonlinearity
  13~(4), 1189--1216.
\newline\urlprefix\url{http://dx.doi.org/10.1088/0951-7715/13/4/311}

\bibitem[{{Fern{\'a}ndez Bonder} et~al.(2002){Fern{\'a}ndez Bonder}, Groisman,
  and Rossi}]{FGRpams02}
{Fern{\'a}ndez Bonder}, J., Groisman, P., Rossi, J.~D., 2002. On numerical
  blow-up sets. Proc. Amer. Math. Soc. 130~(7), 2049--2055.
\newline\urlprefix\url{http://dx.doi.org/10.1090/S0002-9939-02-06350-5}

\bibitem[{Ferreira et~al.(2002)Ferreira, Groisman, and Rossi}]{FGRmmmas02}
Ferreira, R., Groisman, P., Rossi, J.~D., 2002. Numerical blow-up for a
  nonlinear problem with a nonlinear boundary condition. Math. Models Methods
  Appl. Sci. 12~(4), 461--483.
\newline\urlprefix\url{http://dx.doi.org/10.1142/S021820250200174X}

\bibitem[{Ferreira et~al.(2004)Ferreira, Groisman, and Rossi}]{FGRnmpde04}
Ferreira, R., Groisman, P., Rossi, J.~D., 2004. Numerical blow-up for the
  porous medium equation with a source. Numer. Methods Partial Differential
  Equations 20~(4), 552--575.
\newline\urlprefix\url{http://dx.doi.org/10.1002/num.10103}

\bibitem[{Filippas and Kohn(1992)}]{FKcpam92}
Filippas, S., Kohn, R.~V., 1992. Refined asymptotics for the blowup of
  {$u_t-\Delta u=u^p$}. Comm. Pure Appl. Math. 45~(7), 821--869.
\newline\urlprefix\url{http://dx.doi.org/10.1002/cpa.3160450703}

\bibitem[{Filippas and Liu(1993)}]{FLaihn93}
Filippas, S., Liu, W.~X., 1993. On the blowup of multidimensional semilinear
  heat equations. Ann. Inst. H. Poincar{\'e} Anal. Non Lin{\'e}aire 10~(3),
  313--344.

\bibitem[{Friedman(1965)}]{Friams65}
Friedman, A., 1965. Remarks on nonlinear parabolic equations. In: Proc.
  {S}ympos. {A}ppl. {M}ath., {V}ol. {XVII}. Amer. Math. Soc., Providence, R.I.,
  pp. 3--23.

\bibitem[{Fujita(1966)}]{FUJsut66}
Fujita, H., 1966. On the blowing up of solutions of the {C}auchy problem for
  {$u_{t}=\Delta u+u^{1+\alpha }$}. J. Fac. Sci. Univ. Tokyo Sect. I 13,
  109--124 (1966).

\bibitem[{Galaktionov and Posashkov(1985)}]{GPans85}
Galaktionov, V.~A., Posashkov, S.~A., 1985. The equation
  {$u_t=u_{xx}+u^\beta$}. {L}ocalization, asymptotic behavior of unbounded
  solutions. Akad. Nauk SSSR Inst. Prikl. Mat. Preprint~(97), 30.

\bibitem[{Galaktionov and Posashkov(1986)}]{GPans86}
Galaktionov, V.~A., Posashkov, S.~A., 1986. Asymptotics of the process of
  nonlinear heat conduction with absorption in the case of a critical value of
  the parameter. Akad. Nauk SSSR Inst. Prikl. Mat. Preprint~(71), 25.

\bibitem[{Giga and Kohn(1985)}]{GKcpam85}
Giga, Y., Kohn, R.~V., 1985. Asymptotically self-similar blow-up of semilinear
  heat equations. Comm. Pure Appl. Math. 38~(3), 297--319.
\newline\urlprefix\url{http://dx.doi.org/10.1002/cpa.3160380304}

\bibitem[{Giga and Kohn(1987)}]{GKiumj87}
Giga, Y., Kohn, R.~V., 1987. Characterizing blowup using similarity variables.
  Indiana Univ. Math. J. 36~(1), 1--40.
\newline\urlprefix\url{http://dx.doi.org/10.1512/iumj.1987.36.36001}

\bibitem[{Giga and Kohn(1989)}]{GKcpam89}
Giga, Y., Kohn, R.~V., 1989. Nondegeneracy of blowup for semilinear heat
  equations. Comm. Pure Appl. Math. 42~(6), 845--884.
\newline\urlprefix\url{http://dx.doi.org/10.1002/cpa.3160420607}

\bibitem[{Groisman(2006)}]{GROcomp06}
Groisman, P., 2006. Totally discrete explicit and semi-implicit {E}uler methods
  for a blow-up problem in several space dimensions. Computing 76~(3-4),
  325--352.
\newline\urlprefix\url{http://dx.doi.org/10.1007/s00607-005-0136-0}

\bibitem[{Groisman and Rossi(2001)}]{GRjcam01}
Groisman, P., Rossi, J.~D., 2001. Asymptotic behaviour for a numerical
  approximation of a parabolic problem with blowing up solutions. J. Comput.
  Appl. Math. 135~(1), 135--155.
\newline\urlprefix\url{http://dx.doi.org/10.1016/S0377-0427(00)00571-9}

\bibitem[{Herrero and Vel{\'a}zquez(1992{\natexlab{a}})}]{HVasps92}
Herrero, M.~A., Vel{\'a}zquez, J. J.~L., 1992{\natexlab{a}}. Comportement
  g\'en\'erique au voisinage d'un point d'explosion pour des solutions
  d'\'equations paraboliques unidimensionnelles. C. R. Acad. Sci. Paris S\'er.
  I Math. 314~(3), 201--203.

\bibitem[{Herrero and Vel{\'a}zquez(1992{\natexlab{b}})}]{HVasnsp92}
Herrero, M.~A., Vel{\'a}zquez, J. J.~L., 1992{\natexlab{b}}. Generic behaviour
  of one-dimensional blow up patterns. Ann. Scuola Norm. Sup. Pisa Cl. Sci. (4)
  19~(3), 381--450.
\newline\urlprefix\url{http://www.numdam.org/item?id=ASNSP_1992_4_19_3_381_0}

\bibitem[{Herrero and Vel{\'a}zquez(1993)}]{HVaihn93}
Herrero, M.~A., Vel{\'a}zquez, J. J.~L., 1993. Blow-up behaviour of
  one-dimensional semilinear parabolic equations. Ann. Inst. H. Poincar{\'e}
  Anal. Non Lin{\'e}aire 10~(2), 131--189.

\bibitem[{Hirota and Ozawa(2006)}]{HKjcam06}
Hirota, C., Ozawa, K., 2006. Numerical method of estimating the blow-up time
  and rate of the solution of ordinary differential equations---an application
  to the blow-up problems of partial differential equations. J. Comput. Appl.
  Math. 193~(2), 614--637.
\newline\urlprefix\url{http://dx.doi.org/10.1016/j.cam.2005.04.069}

\bibitem[{Hocking et~al.(1972)Hocking, Stewartson, Stuart, and
  Brown}]{HSBjfm72}
Hocking, L.~M., Stewartson, K., Stuart, J.~T., Brown, S.~N., 1 1972. A
  nonlinear instability burst in plane parallel flow. Journal of Fluid
  Mechanics 51, 705--735.
\newline\urlprefix\url{http://journals.cambridge.org/article_S0022112072001326}

\bibitem[{Levine(1973)}]{LEVarma73}
Levine, H.~A., 1973. Some nonexistence and instability theorems for solutions
  of formally parabolic equations of the form {$Pu_{t}=-Au+{F}(u)$}. Arch.
  Rational Mech. Anal. 51, 371--386.

\bibitem[{Masmoudi and Zaag(2008)}]{MZjfa08}
Masmoudi, N., Zaag, H., 2008. Blow-up profile for the complex
  {G}inzburg-{L}andau equation. J. Funct. Anal. 255~(7), 1613--1666.
\newline\urlprefix\url{http://dx.doi.org/10.1016/j.jfa.2008.03.008}

\bibitem[{Merle and Zaag(1997)}]{MZdm97}
Merle, F., Zaag, H., 1997. Stability of the blow-up profile for equations of
  the type {$u_t=\Delta u+\vert u\vert ^{p-1}u$}. Duke Math. J. 86~(1),
  143--195.
\newline\urlprefix\url{http://dx.doi.org/10.1215/S0012-7094-97-08605-1}

\bibitem[{Merle and Zaag(1998{\natexlab{a}})}]{MZcpam98}
Merle, F., Zaag, H., 1998{\natexlab{a}}. Optimal estimates for blowup rate and
  behavior for nonlinear heat equations. Comm. Pure Appl. Math. 51~(2),
  139--196.
\newline\urlprefix\url{http://dx.doi.org/10.1002/(SICI)1097-0312(199802)}

\bibitem[{Merle and Zaag(1998{\natexlab{b}})}]{MZgfa98}
Merle, F., Zaag, H., 1998{\natexlab{b}}. Refined uniform estimates at blow-up
  and applications for nonlinear heat equations. Geom. Funct. Anal. 8~(6),
  1043--1085.
\newline\urlprefix\url{http://dx.doi.org/10.1007/s000390050123}

\bibitem[{Nakagawa(1975/76)}]{NAKamo76}
Nakagawa, T., 1975/76. Blowing up of a finite difference solution to
  {$u_{t}=u_{xx}+u_{2}.$}. Appl. Math. Optim. 2~(4), 337--350.

\bibitem[{Nakagawa and Ushijima(1977)}]{NUna77}
Nakagawa, T., Ushijima, T., 1977. Finite element analysis of the semi-linear
  heat equation of blow-up type. in Topics in Numerical Analysis, J. J. H.
  Miller, ed., Academic Press.

\bibitem[{N'gohisse and Boni(2011)}]{NBams11}
N'gohisse, F.~K., Boni, T.~K., 2011. Numerical blow-up for a nonlinear heat
  equation. Acta Math. Sin. (Engl. Ser.) 27~(5), 845--862.

\bibitem[{Philippe(1996)}]{SOUmmas96}
Philippe, S., 1996. Finite time blow-up for a non-linear parabolic equation
  with a gradient term and applications. Mathematical Methods in the Applied
  Sciences 19~(16), 1317--1333.
\newline\urlprefix\url{http://dx.doi.org/10.1002/(SICI)1099-1476}

\bibitem[{Popp et~al.(1998)Popp, Stiller, Kuznetsov, and Kramer}]{POPphd98}
Popp, S., Stiller, O., Kuznetsov, E., Kramer, L., 1998. The cubic complex
  ginzburg-landau equation for a backward bifurcation. Physica D: Nonlinear
  Phenomena 114~(1-2), 81 -- 107.
\newline\urlprefix\url{http://www.sciencedirect.com/science/article/pii/S016727899700170X}

\bibitem[{Quittner and Souplet(2007)}]{QSbook07}
Quittner, P., Souplet, P., 2007. Superlinear parabolic problems. Birkh{\"a}user
  Advanced Texts: Basler Lehrb{\"u}cher. [Birkh{\"a}user Advanced Texts: Basel
  Textbooks]. Birkh{\"a}user Verlag, Basel, blow-up, global existence and
  steady states.

\bibitem[{Souplet et~al.(1996)Souplet, Tayachi, and Weissler}]{STWiumj96}
Souplet, P., Tayachi, S., Weissler, F.~B., 1996. Exact self-similar blow-up of
  solutions of a semilinear parabolic equation with a nonlinear gradient term.
  Indiana Univ. Math. J. 45~(3), 655--682.
\newline\urlprefix\url{http://dx.doi.org/10.1512/iumj.1996.45.1197}

\bibitem[{Ushijima(2000)}]{USHprims00}
Ushijima, T.~K., 2000. On the approximation of blow-up time for solutions of
  nonlinear parabolic equations. Publ. Res. Inst. Math. Sci. 36~(5), 613--640.
\newline\urlprefix\url{http://dx.doi.org/10.2977/prims/1195142812}

\bibitem[{Vel{\'a}zquez(1992)}]{VELcpde92}
Vel{\'a}zquez, J. J.~L., 1992. Higher-dimensional blow up for semilinear
  parabolic equations. Comm. Partial Differential Equations 17~(9-10),
  1567--1596.
\newline\urlprefix\url{http://dx.doi.org/10.1080/03605309208820896}

\bibitem[{Vel{\'a}zquez(1993)}]{VELtams93}
Vel{\'a}zquez, J. J.~L., 1993. Classification of singularities for blowing up
  solutions in higher dimensions. Trans. Amer. Math. Soc. 338~(1), 441--464.
\newline\urlprefix\url{http://dx.doi.org/10.2307/2154464}

\bibitem[{Zaag(1998)}]{ZAAihn98}
Zaag, H., 1998. Blow-up results for vector-valued nonlinear heat equations with
  no gradient structure. Ann. Inst. H. Poincar{\'e} Anal. Non Lin{\'e}aire
  15~(5), 581--622.
\newline\urlprefix\url{http://dx.doi.org/10.1016/S0294-1449(98)80002-4}

\end{thebibliography}
\end{document}